\documentclass{article}

\usepackage{geometry}
\usepackage{tikz}
\usepackage{tikz}
\usepackage{pgfplots}
\usepgfplotslibrary{groupplots}
\pgfplotsset{compat=1.18}
\usepackage{graphicx} 
\usepackage{epstopdf} 
\usepackage[colorlinks, hypertexnames=false]{hyperref}
\hypersetup{linkcolor=blue, citecolor=red} 

\usepackage{bm}
\usepackage{amsmath} 
\usepackage{amssymb} 
\usepackage{stmaryrd} 
\usepackage{multirow}

\usepackage{amsthm} 
\usepackage{mathtools} 

\usepackage[dvipsnames]{xcolor}
\definecolor{red}{rgb}{1.00,0.00,0.00}
\numberwithin{equation}{section}
\newtheorem{theorem}{Theorem}[section]
\newtheorem{lemma}{Lemma}[section]
\newtheorem{remark}{Remark}[section]
\newtheorem{definition}{Definition}[section]

\newcommand{\normmm}[1]{{\left\vert\kern-0.25ex\left\vert
		\kern-0.25ex\left\vert #1
		\right\vert\kern-0.25ex\right\vert\kern-0.25ex\right\vert}}

\newcommand{\jump}[1]{\ensuremath{\left\llbracket #1\right\rrbracket}}
\newcommand{\tdiv}{\mathrm{div}}
\newcommand{\tcurl}{\bm{\mathrm{curl}}}

\newcommand{\vertiii}[1]{{\left\vert\kern-0.25ex\left\vert\kern-0.25ex\left\vert #1
		\right\vert\kern-0.25ex\right\vert\kern-0.25ex\right\vert}}

\newcommand{\bFa}{{\overline{\mathcal{F}}_{\ba}}}
\newcommand{\Fma}{{\mathcal{F}}_{\setminus\ba}}

\newcommand{\n}{{\boldsymbol{n}}}
\newcommand{\ba}{{\mathbf{a}}}

\newcommand{\Fall}{\mathcal{F}_h}
\newcommand{\Fb}{\mathcal{F}_h\upb}

\newcommand{\Fint}{\mathcal{F}_h\upi}

\newcommand{\Fa}{\mathcal{F}_{\ba}}

\newcommand{\dK}{\partial K}

\newcommand{\mesh}{\mathcal{T}_h}

\newcommand{\vertice}{\mathcal{V}_h}

\newcommand{\mesha}{\mathcal{T}_{\ba}}
\newcommand{\oma}{\omega_{\ba}}

\newcommand{\upi}{^{\mathrm{int}}}
\newcommand{\upb}{^{\mathrm{bd}}}

\newcommand{\mBS}{\mathcal{I}_{{\rm mBS}}}

\newcommand{\mKM}{\mathcal{I}_{{\rm {mKM}}}^{p,\ba}}
\newcommand{\sumj}{\sum\limits_{j\in\{1{:}d\}}}
\newcommand{\curlrw}[1]{\underline{\nabla}_{\mathrm{rw}} {\times} {#1}}
\newcommand{\divrw}[1]{\nabla_{\mathrm{rw}} {\cdot} {#1}}

\newcommand{\curl}[1]{\nabla {\times} {#1}}
\DeclareMathOperator*{\argmin}{arg\,min}

\begin{document}

	\title{Pressure-robust $hp$-a posteriori error estimates of $ \boldsymbol{H}(\mathrm{div})$-conforming discontinuous Galerkin methods for \\the Stokes equations}

\author{Zhaonan Dong\footnotemark[1], \quad  Zuodong Wang\footnotemark[2], \quad Lina Zhao\footnotemark[3]}

\footnotetext[1]{Inria, 48 rue Barrault, 75647 Paris, France and CERMICS, CNRS, ENPC, Institut Polytechnique de Paris,  6 \& 8 avenue B.~Pascal, 77455 Marne-la-Vall\'{e}e, France. ({zhaonan.dong@inria.fr})}

\footnotetext[2]{School of Mathematical Sciences; Eastern Institute of Technology, Ningbo, Zhejiang, 315200, China.  ({zdwang@eitech.edu.cn}). }

\footnotetext[3]{Department of Mathematics, City University
	of Hong Kong, Kowloon Tong, Hong Kong SAR, China.  ({linazha@cityu.edu.hk}). }

\maketitle
\begin{abstract}
We devise and analyze a pressure-robust residual-based $hp$-a posteriori error estimator for $\boldsymbol{H}(\mathrm{div})$-conforming discontinuous Galerkin (dG) methods for the Stokes problem on two- and three-dimensional polytopal Lipschitz domains. The estimator provides an upper bound and a local lower bound for the velocity error in the energy norm, both robust with respect to the viscosity and independent of the pressure. Our analysis relies on a decomposition of the error into conforming and nonconforming parts. The nonconforming error is bounded using a partition-of-unity framework combined with local Helmholtz decompositions on vertex patches. The conforming error is analyzed by means of the generalized Bogovskiĭ operator of \cite{CostabelMcIntosh2010} in both two and three dimensions, yielding two pressure-independent residual-based estimators associated with different interpolation operators. In the first approach, the upper bound for the conforming error consists of five error indicators and a data oscillation term. Four of these indicators exhibit $p$-optimal scaling, while the remaining one is suboptimal by a factor of $p^{1/2}$. In the second approach, the upper bound involves only two residual indicators together with the data oscillation term, at the expense of losing one order in $p$. Moreover, a pressure-robust local lower bound is established using $H^2$-bubble functions inspired by techniques developed for fourth-order PDEs. Numerical results in two and three dimensions confirm the reliability, efficiency, and pressure-robustness of the proposed estimators.
\end{abstract}



%

\section{Introduction}
Stokes problems arise in many scientific and engineering applications, including incompressible fluid dynamics, porous media flow, microfluidics, biological transport, and fluid--structure interaction. Let $\Omega$ be a polygonal or polyhedral domain in $\mathbb{R}^d$, $d=2,3$, with Lipschitz boundary $\partial\Omega$. The Stokes equations are given by
\begin{alignat*}{2}
- \mu \Delta \bm{u} +\nabla p&=\bm{f} &&\quad \mbox{in}\; \Omega,\\
	\nabla{\cdot}\bm{u}&=0&&\quad \mbox{in}\; \Omega,\\
	\bm{u}&=\bm{0}&&\quad \mbox{on}\; \partial \Omega ,
\end{alignat*}
where the condition $\int_{\Omega} p\;dx =0$ is imposed to ensure the uniqueness of the pressure and, consequently, the unique solvability of the problem.

Among the many discretisation techniques for the Stokes equations, pressure-robust methods have attracted considerable attention in recent years. A discretisation is said to be pressure-robust if the velocity error is independent of the pressure (or, equivalently, of irrotational forcing terms). Classical mixed finite element methods, although satisfying the discrete inf--sup condition, typically yield velocity error estimates containing a pressure contribution scaled by $\mu^{-1}$. As highlighted in the comprehensive review~\cite{VACMG17}, this pressure dependence can significantly deteriorate the accuracy of the velocity approximation when the pressure is large or exhibits complex behaviour, even for simple benchmark problems such as the no-flow and stationary vortex cases. Pressure-robust methods overcome this limitation by enforcing a stronger form of the incompressibility constraint, either through exactly divergence-free velocity spaces (e.g., Scott--Vogelius elements and $\bm{H}(\mathrm{div})$-conforming methods) or through suitable reconstructions of the test functions. These methods offer several important advantages, including velocity error estimates that are independent of the pressure and improved accuracy for coupled problems, such as natural convection and passive scalar transport, as well as for flows dominated by Coriolis effects or large irrotational forces.

While considerable effort has been devoted to the \textit{a priori} analysis of pressure-robust discretizations, and their advantages are now well understood, we refer to \cite{LinkeMerdon16,AinsworkthParker1,VACMG17,LiuLiuPego2007} for comprehensive discussions. In contrast, \textit{a posteriori} error estimation remains largely unexplored.  A posteriori error estimators play a fundamental role in adaptive mesh refinement strategies and in providing reliable error control for practical simulations. Classical techniques for deriving a posteriori error estimators for the Stokes equations typically involve pressure-dependent contributions. Early contributions focused on residual-based estimators for conforming finite element discretisations \cite{Verfurth89}. These techniques were subsequently extended to nonconforming methods \cite{DariDuranPadra1995} and to more general incompressible flow models, including the Oseen equations \cite{AinsworthOden1997}. Further developments introduced refined residual-based frameworks \cite{CarstensenFunken2001}, equilibrated estimators for nonconforming finite element methods \cite{DorflerAinsworth2005}, and reliable and efficient estimators for mixed pseudostress--velocity formulations \cite{CarstensenKimPark2011}. More recently, guaranteed error bounds using equilibrated stress reconstruction for a broad class of Stokes discretisations was proposed in \cite{Hannukainen12}. A rigorous $hp$-analysis of a residual-type a posteriori error estimator for a hybrid dG method was presented in \cite{EggerWaluga2013}. Despite these substantial advances, when the underlying discretisation is pressure-robust, the resulting estimators generally fail to inherit this property and may depend explicitly on the viscosity parameter $\mu$. In pressure-dominated regimes, such estimators may overestimate the velocity error and fail to accurately identify regions requiring mesh refinement. This limitation highlights the necessity of developing genuinely pressure-robust a posteriori error estimators.

There are a few works that explicitly address pressure-robust a posteriori error estimators for the Stokes equations. In \cite{KS14}, a stream-function formulation was employed to derive a pressure-independent error estimator for $\bm H(\tdiv)$-conforming methods. However, the analysis was restricted to two-dimensional domains, and no local efficiency result was established. More recently, a reliable and efficient a posteriori error estimator for the stream function vorticity formulation of the incompressible Navier--Stokes equations in two dimensions was derived in \cite{CG26}.
A class of pressure-robust conforming finite element methods equipped with velocity reconstruction operators was analysed in \cite{LCJ19}. The resulting estimator was shown to be pressure-robust with respect to the reliability estimate; however, the corresponding local efficiency bound still involved pressure-dependent contributions, and the analysis was again confined to two-dimensional settings. Subsequently, a guaranteed error estimator for the velocity error was developed in \cite{LC22} through the construction of an equilibrated flux based on a mass-conserving mixed stress discretisation \cite{JLJ20}. Both global reliability and local efficiency were established for the resulting estimator. To the best of our knowledge, pressure-robust $hp$-a posteriori error estimators for the Stokes problem have not yet been investigated, even in two dimensions. Consequently, it is of great importance and interest to develop a rigorous pressure-robust $hp$-a posteriori error estimator that features both a provable upper bound and a local lower bound.

In this work, we derive residual $hp$-\emph{a posteriori} error estimates for $\bm{H}(\tdiv)$-conforming dG methods applied to the Stokes problem posed on two- and three-dimensional polytopal Lipschitz domains. To the best of our knowledge, this is the first work to employ a partition of unity combined with local Helmholtz decompositions on vertex patches to control the nonconforming error in the $hp$-setting. A key contribution of this work is the derivation of a novel stability estimate for the local Helmholtz decomposition, together with a carefully designed patchwise Stokes reconstruction that lifts the product of an $\bm{H}(\tdiv)$-conforming piecewise polynomial and a hat function into $\mathbf{H}^1_0(\oma)$ on the vertex patch $\oma$ associated with $\mathbf{a}$ while preserving the divergence exactly. We believe that this result is of independent interest, since the underlying decomposition can be applied to a broad class of nonconforming finite element methods, including HDG methods. Moreover, the use of local Helmholtz decompositions, as opposed to a global decomposition, eliminates any dependence of the resulting estimates on the number of holes in the domain $\Omega$.

The second main contribution of this work is the derivation of two pressure-robust upper bounds for the conforming error. We emphasize that preserving the divergence-free constraint is particularly challenging in the construction of pressure-robust error estimators. To overcome this difficulty, our analysis relies on the generalized Bogovskii operator introduced in~\cite{CostabelMcIntosh2010}. Specifically, for any function $\bm{v}\in \bm{H}^1_0(\Omega)$ satisfying $\nabla{\cdot}\bm{v}=0$, there exists a function $\bm{\zeta}\in \bm{H}^2_0(\Omega)$ such that $\curl{\bm{\zeta}}=\bm{v}$ on Lipschitz domains with simple topology. This seminal result provides the foundation for constructing interpolation operators that preserve the divergence-free condition. The first approach is based on the classical Babu\v{s}ka--Suri interpolation operator~\cite{babuskasurihpversionFEMwithquasiuniformmesh}. The resulting pressure-robust upper bound involves five error indicators together with a data oscillation term. Four of the indicators are $hp$-optimal, while the jump penalty term exhibits a suboptimality of order $p^{1/2}$. The second approach relies on the N\'ed\'elec projection-based interpolation operator~\cite{Melenkhp}, whose key features are its commuting property and orthogonality. The resulting upper bound for the conforming error is suboptimal by at most one order in $p$. Its main advantage is its remarkable simplicity: the estimator involves only two error indicators together with a data oscillation term.

Another contribution of this work is the derivation of a pressure-robust local lower bound, improving upon the existing result in~\cite{JLJ20}. The pressure-robustness of the local lower bound is established by means of a bespoke $C^1$-bubble function inspired by techniques developed for fourth-order problems.
The proposed methodology is not restricted to $\bm{H}(\mathrm{div})$-conforming dG methods and has broader implications for pressure-robust discretisations. In particular, it can be employed to derive both upper and lower error bounds for a wide range of pressure-robust finite element methods, including the Scott--Vogelius element~\cite{Zhang05}, divergence-free elements using rational functions \cite{GuzmanNeilan2014,GuzmanNeilan2014b}, the mass-conserving mixed stress method~\cite{JLJ20}, staggered dG methods~\cite{ZCL20}, and pressure-robust schemes based on velocity reconstruction operators~\cite{LLMJ17,ZPC22}.

The remainder of the paper is organized as follows. In Section~\ref{sec:main}, we introduce the weak formulation of the model problem and recall the main $hp$-approximation tools and interpolation operators employed in the analysis. Section~\ref{sec: H(div)-DG} presents the $\bm{H}(\tdiv)$-conforming dG method. In Section~\ref{sec: Local Helmholtz Decomposition}, we establish a novel stability estimate for local Helmholtz decompositions on vertex patches; the main result is Lemma~\ref{lemma:helmholtz}. Section~\ref{sec:apost} is devoted to the derivation of the residual-based $hp$-\emph{a posteriori} upper and local lower error bounds, whose main results are stated in Theorems~\ref{theorem: a posterioir error bound} and~\ref{theorem: lower bound}. Numerical experiments illustrating the theoretical findings are presented in Section~\ref{sec:Numerical example}. Finally, Section~\ref{sec:proofs} collects several technical proofs used in the $hp$-\emph{a posteriori} error analysis.

\section{Weak form and analysis tools}
\label{sec:main}
In this section, we introduce essential notation at the continuous and discrete levels, formulate the weak problem, and recall some useful results from the literature.

\subsection{Basic notation and weak formulation}
We adopt standard notation for Lebesgue and Sobolev spaces. Let $S \subset \mathbb{R}^d$, $d \in \{2,3\}$, be an open, bounded, Lipschitz set. For scalar-, vector-, or tensor-valued fields, we denote the $L^2$-inner product and norm as $(\bullet,\bullet)_S$ and $\|\bullet\|_S$, respectively.
Standard notation of Sobolev spaces is used in this work. We employ boldface font to denote vector-valued fields, and we use an additional underline notation, e.g., $\underline{\bm{A}}$, for tensor-valued fields. The (weak) gradient of a scalar-valued function $v$ is denoted as $\nabla v$. Let $\bm{n}_S$ denote the unit outward normal vector on the boundary $\partial S$ of~$S$. We consider the Hilbert Sobolev spaces
\begin{align*}
	\bm{V}:&= \bm{H}^1_{0}(\Omega),\quad
	Q:=L^2_0 (\Omega), \quad
	\bm{H}(\tcurl; \Omega):=\{\bm{v}\in \bm{L}^2(\Omega);\nabla {\times} \bm{v} \in \bm{L}^2(\Omega)\},\\
	\bm{H}(\tdiv;\Omega):&=\{\bm{v}\in \bm{L}^2(\Omega);\nabla{\cdot}\bm{v}\in L^2(\Omega)\},
	\bm{H}(\tdiv^0;\Omega):=\{\bm{v}\in \bm{L}^2(\Omega);\nabla{\cdot}\bm{v}=0\}.
\end{align*}
In addition, we follow the convention by setting $\bm{H}_0(\tcurl;\Omega):=\{\bm{v}\in \bm{H}(\tcurl;\Omega); \bm{v}{\times} \bm{n}_{\Omega} = \bm{0}\}$ and $\bm{H}_0(\tdiv;\Omega):=\{\bm{v}\in \bm{H}(\tdiv;\Omega); \bm{v}{\cdot} \bm{n}_{\Omega} = 0\}$. We define
\begin{align*}
	a(\bm{u},\bm{v}):=(\mu \nabla \bm{u}, \nabla \bm{v})_{\Omega},\quad
	b(\bm{v},p):=-(p,\nabla{\cdot} \bm{v})_{\Omega}.
\end{align*}
The weak formulation reads: Find $(\bm{u},p)\in \bm{V}{\times} Q$ such that for all $(\bm{v},q)\in \bm{V} {\times} Q$,
\begin{align*}
	a(\bm{u},\bm{v})+b(\bm{v},p) &=(\bm{f},\bm{v})_{\Omega},\\
    -b(\bm{u},q) &=0.
\end{align*}
We define the bilinear form
\begin{align*}
	B((\bm{u},p);(\bm{v},q)):&=	a(\bm{u},\bm{v})+b(\bm{v},p) -b(\bm{u},q).
\end{align*}
Then, the weak formulation can be recast into the following equivalent form: Find $(\bm{u},p)\in \bm{V}{\times} Q$ such that
\begin{align}
	B((\bm{u},p);(\bm{v},q))= (\bm{f},\bm{v})_{\Omega}, \qquad \forall (\bm{v},q)\in \bm{V}{\times} Q .\label{eq:weak-B}
\end{align}
The well-posedness of the above problem can be found, for example, in \cite{Girault86}. As the solution $\bm{u}\in \bm{H}^1_{0}(\Omega) \cap \bm{H}(\tdiv^0;\Omega)$, we can eliminate pressure and reformulate \eqref{eq:weak-B} as follows:
\begin{align}
	a(\bm{u},\bm{v})= (\bm{f},\bm{v})_{\Omega}, \qquad \forall \bm{v}\in \bm{H}^1_{0}(\Omega) \cap \bm{H}(\tdiv^0;\Omega).\label{eq:weak-div-free}
\end{align}

\subsection{ Mesh and discrete setting}

Let $\mesh$ denote a conforming simplicial mesh of the domain $\Omega$. A generic mesh cell is denoted by $K \in \mesh$, and we let $h_K$ be the diameter of the element $K$. For any generic mesh vertex $\ba$ in the set of mesh vertices $\vertice$, we denote $\oma$ as its patch, i.e., the collection of cells sharing this vertex. In addition, we denote $\mesha$ as the subset of $\mesh$ such that for each cell $K\in\mesha$, $\ba\in \Bar{K}$. The mesh faces are collected in the set $\Fall$. We denote by $\Fint$ the set of all interior faces and by $\Fb$ the set of all boundary faces. Hence, $\Fall = \Fint \cup \Fb$. Next, for each face $F$, we use $h_F$ to denote the diameter of $F$, and $\n_F$ to represent the unit normal vector of $F$
pointing from $K_1$ to $K_2$, where $K_1$ and $K_2$ are the cells sharing the common face $F$.
When there is no confusion, we use $\n$ to simplify the notation.
For all $s>\frac12$, we define the broken Sobolev spaces $H^s(\mesh;\mathbb{R}^d):= \{w\in L^2(\Omega;\mathbb{R}^d)\:|\:w_K:=w|_K\in H^s(K;\mathbb{R}^d), ~ \forall K\in \mesh\}$, $d\in\{2,3\}$. 
For each interior face $F\in \Fint$, we define the jump and average operators for function $v\in H^s(\mesh;\mathbb{R}^d)$ over $F=\partial K_1 \cap \partial K_2$ as
\begin{align*}
	\jump{v}_{F}:=v_{K_1}|_F -v_{K_2}|_F\quad \mbox{and}\quad \{v\}_{F} :=\frac{v_{K_1}|_F+v_{K_2}|_F}{2},
\end{align*}
and $\bm{n}_F$ points from $K_1$ to $K_2$. For the boundary faces $F=  \partial K \cap \partial \Omega$, we simply define $\jump{v}_{F}:=v_{K_1}|_F$ and $\{v\}_{F}:=v_{K_1}|_F$.  Finally, we define the broken gradient $\nabla_h$, broken Laplacian $\Delta_h$, and broken curl $\nabla_h{\times}$ as the cellwise gradient, Laplacian, and curl operators acting on $H^{1}(\mesh;\mathbb{R}^q)$, $H^{2}(\mesh;\mathbb{R}^q)$, and $\bm{H}(\tcurl;\mesh)$, respectively.

Let $p\geq 0$ denote the polynomial degree, we use $\mathbb{P}_p(K)$ and $\mathbb{P}_p(F)$ to represent the polynomial functions defined on $K$ and $F$ whose order is less than or equal to $p$.  We define the $L^2$-orthogonal projection $\Pi^p_K$ onto $\mathbb{P}_p(K)$. At some occasions, we also consider $p\le -1$, in which case $\mathbb{P}_{p}(K):=\{0\}$ and $\Pi^{p}_K$ identically maps to the zero function.  Next,  $\bm{P}_p$ and $\bm{\underline{P}}_p$ represent the space composed of vector- and tensor-valued polynomials of total degree at most $p$, respectively. We consider the following vector-valued broken polynomial spaces:
\begin{align*}
	\bm{P}_p(\mesh):&=\{\bm{v}_h\in \bm{L}^2(\Omega); \bm{v}_h|_{K} \in [\mathbb{P}_p(K)]^d, \forall K\in \mesh\},\\
	\bm{RT}_p(\mesh):&=\{\bm{v}_h\in \bm{L}^2(\Omega); \bm{v}_h|_{K} \in [\mathbb{P}_p(K)]^d + \bm{x} \tilde{\mathbb{P}}_p(K), \forall K\in \mesh\},\\
  \ \bm{N}_{p}(\mathcal{T}_h):&=\{\bm{v}_h\in \bm{L}^2(\Omega);\bm{v}_h|_K\in [\mathbb{P}_p(K)]^d + \bm{x} {\times} \tilde{\mathbb{P}}_p(K), \forall K\in \mesh \},
\end{align*}
where $\tilde{\mathbb{P}}_p(K)$ denotes homogeneous polynomials of degree $p$. Similarly, tensor-valued broken polynomial space is defined as follow:
\begin{align*}
	\underline{\bm{P}}_p(\mesh):=\{\underline{\bm{w}}_h\in \underline{\bm{L}}^2(\Omega), {\underline{\bm{w}}_h}_{|K}\in \underline{P}_p(K),K\in \mesh\}.
\end{align*}

\subsection{Analysis tools}\label{Analysis tools}
We briefly review the main $hp$-analysis tools used in this work.
We use the symbol $C$ to denote any positive generic constant whose value can change at each occurrence as long as it is independent of the mesh size $h$ and the  polynomial degree $p$. The value of $C$ can  depend on the mesh shape-regularity and the space dimension $d$. In the case that the generic constant depends on the domain $\Omega$ or polynomial degree, we denote it by $C(\Omega)$ and $C(p)$, respectively.

\begin{lemma}[$hp$-discrete trace inequality]\label{sharp_inverse}
	Let $v \in \mathbb{P}_p({K})$ with $p\geq0$, for all $K \in \mesh$. Then, the following estimate holds:
	\begin{subequations}
		\begin{align}\label{eq: discrete trace}
			\|v\|_{\dK} \le C \bigg(\dfrac{(p+1)^2}{h_K}\bigg)^{\frac{1}{2}}\|v\|_{K}.
		\end{align}
	\end{subequations}
\end{lemma}

\begin{proof}
	The proof of~\eqref{eq: discrete trace} can be found in \cite{Warburton03}.
\end{proof}

\begin{lemma}[$hp$-inverse estimate]\label{lemma: Inverse inequality}
	The following holds for all $v\in \mathbb{P}_p(K)$, for all $K\in\mesh$, and all $p\ge0$,
	\begin{align}
		\|\nabla v\|_{K}\le C \frac{p^2}{h_K} \|v\|_{K}.\label{eq:inverse}
	\end{align}
\end{lemma}

\begin{proof}
	A proof can be found in  \cite[Theorem~4.76]{Schwab98}.
\end{proof}

\begin{lemma}[Local $L^2$-orthogonal projection] The following holds for all $v\in H^1(K)$, $K\in \mathcal{T}_h$ and $p\geq 0$,
	\begin{align}
		\|v-\Pi_K^p(v)\|_{\partial K}\le C \left(\frac{h_K}{p+1}\right)^{\frac{1}{2}}\|\nabla v\|_{K},\label{eq:L2projection-trace}
	\end{align}
and
	\begin{align}
		\|\nabla \Pi_K^p(v)\|_{K}\le C p^{\frac12}\|\nabla v\|_{K}.\label{eq:L2projection-stability}
	\end{align}
\end{lemma}

\begin{proof}
	A proof of \eqref{eq:L2projection-trace} can be found in \cite[Theorem~4.1]{Chernov12}, whereas the proof of \eqref{eq:L2projection-stability} can be found in \cite[Appendix C]{CarGraTra24}.
\end{proof}


\begin{lemma}[Local modified-Karkulik--Melenk operator]\label{lemma: Local hp-KM}
	There exists a constant $C$ that only depends on the mesh shape-regularity such that, for all $p\geq 1$, there exists an operator $\mKM: H^1(\oma)\rightarrow \mathbb{P}_p(\mesha)\cap H^1(\oma)$,  such that for all $v\in H^1(\oma)$, and all $K\in\mesha$,
	\begin{align}
		\bigg(\frac{p}{h_K}\bigg)^2	\|v-\mKM (v)\|_{K}^2+\frac{p}{h_K}\|v - \mKM (v)\|_{\partial K}^2+\|\nabla \mKM (v)\|_{K}^2 \leq C \|\nabla v\|_{\oma}^2.  \label{eq: hp-approximation}
	\end{align}
	\begin{proof}
		See \cite{Melenk05,KarMel2015} for the original proof. (see also \cite[Corollary 2.5]{dong:hal-04720237}  for using the $H^1$-seminorm on the right-hand side).
	\end{proof}
\end{lemma}

\begin{lemma}[Modified-global Babu\v{s}ka-Suri-operator]\label{lemma: Global BS}
	There exists a constant $C$ such that, for all $p\geq 1$, there exists an operator $\mBS^p :H^1_{0}(\Omega)\cap H^2(\mesh)\rightarrow \mathbb{P}_p(\mesh)\cap H^1_{0}(\Omega)$,  such that for all $v\in H^1_{0}(\Omega)\cap H^2(\mesh)$, and all $K\in\mesh$,
	\begin{align}
		\bigg(\frac{p}{h_K}\bigg)^4	\|v-\mBS^p (v)\|_{K}^2+\bigg(\frac{p}{h_K}\bigg)^{\frac32}\|v-\mBS^p (v)\|_{\partial K}^2 \nonumber\\
        +\bigg(\frac{p}{h_K}\bigg)^{\frac12}\|\nabla(v-\mBS^p (v))\|_{\partial K}^2 +		\|\nabla^2 \mBS^p (v)\|_{K}^2 \leq C \|\nabla^2 v\|_{K}^2.  \label{eq: hp-approximation BS}
	\end{align}
	\begin{proof}
		The original proof of  Babu\v{s}ka-Suri-operator can be found in \cite[Lemma~4.5]{BabuSurioptimalconvergenceestimatepmethods} (see also \cite[Corollary 2.4]{dong:hal-05498158}   for using the $H^2$-seminorm on the right-hand side).
	\end{proof}

\end{lemma}

\begin{lemma}[Global Raviart--Thomas projection-based interpolation]\label{lemma: Global RT}
There exists an operator
$
\bm{I}_{\rm RT}^{p}: \bm{H}^1_0(\Omega)
\rightarrow
\bm{RT}_{p}(\mathcal{T}_h)\cap \bm{H}_0(\tdiv;\Omega),
$
with $p\geq 1$, such that for any
$\bm{v}\in \bm{H}^1_0(\Omega)\cap \bm{H}(\tdiv^0;\Omega)$ and for all $K\in\mesh$, we have
\begin{align}
\|	\bm{v}-\bm{I}_{\rm RT}^{p} (\bm{v})\|_{K}
\leq
C
\left(\frac{h_{K}}{p}\right)
| \bm{v}|_{\bm{H}^1(K)}.
\label{eq:RT-hp}
\end{align}
\end{lemma}

\begin{proof}
The construction of the Raviart--Thomas projection-based interpolation operator can be found in \cite[Theorem~2.10, property (vi)]{Melenkhp}. Since
$\bm{v}\in \bm{H}^1_0(\Omega)\cap \bm{H}(\tdiv^0;\Omega)$ satisfies
$\nabla{\cdot} \bm{v}=0$, the divergence contribution appearing in the general approximation estimate vanishes. Consequently, the interpolation error satisfies the $p$-optimal bound \eqref{eq:RT-hp}.
\end{proof}

\begin{lemma}[Global N\'ed\'elec projection-based interpolation]\label{lemma: Global Nedelec} There exists an operator $\bm{I}_{\rm N}^{p}:\bm{H}^2_0(\Omega)\rightarrow \bm{N}_{p}(\mathcal{T}_h)\cap \bm{H}_0(\tcurl;\Omega)$, with $p\geq1$, such that for any $\bm{v}\in \bm{H}^2_0(\Omega)$, the following orthogonality properties hold,
	\begin{align}
		(\bm{I}_{\rm N}^{p} (\bm{v})-\bm{v}, \bm{q}_K)_K&=0\quad \forall \bm{q}_K\in \bm{P}_{p-2}(K), ~ \forall K\in \mesh, \label{eq:Nedelec-property3}\\
	(\curl{(\bm{I}_{\rm N}^{p} (\bm{v})-\bm{v})}, \bm{w}_K)_K&=0\quad \forall \bm{w}_K\in \bm{P}_{p-1}(K),~ \forall K\in \mesh. \label{eq:Nedelec-property4}
	\end{align}
In addition, the following commuting relation holds
\begin{equation}\label{eq:Nedelec-commut}
 \curl{ \bm{I}_{\rm N}^{p} (\bm{v})} = \bm{I}_{\rm RT}^{p}(\curl{ \bm{v}}).
\end{equation}
Moreover, for all $K\in\mesh$, we have
	\begin{align}
		\|	\bm{v}-\bm{I}_{\rm N}^{p} (\bm{v})\|_{K}\leq C \left(\frac{h^2_K}{p}\right)  | \bm{v}|_{\bm{H}^2(K)}.\label{eq:Nedelec-hp}
	\end{align}
\end{lemma}

\begin{proof}
The orthogonality conditions defining the projection-based interpolation operator can be found in \cite[Definition~2.2]{Melenkhp} and also see \cite[Section 2.5.3]{Boffibook}. The construction and approximation properties of the N\'ed\'elec projection-based interpolation operator are established in \cite[Theorem~2.10, property~(iii)]{Melenkhp}. In particular, for any $\bm{v}_h\in \bm{N}_{p}(\mathcal{T}_h)\cap \bm{H}_0(\tcurl;\Omega)$, it holds that
\[
\|\bm{v}-\bm{I}_{\rm N}^{p}(\bm{v})\|_{K}
\leq
C p^{-1}\Big(
\|\bm{v}-\bm{v}_h\|_{K}
+
h_K|\bm{v}-\bm{v}_h|_{\bm{H}^1(K)}
+
h_K\|\curl(\bm{v}-\bm{v}_h)\|_{K}
+
h_K^2|\curl(\bm{v}-\bm{v}_h)|_{\bm{H}^1(K)}
\Big).
\]
Since $\bm{v}\in \bm{H}^2_0(\Omega)$, we choose
$
\bm{v}_h=\bm{\mathcal{I}}^1_{\rm c}\bm{v}
\in
\bm{N}_{1}(\mathcal{T}_h)\cap \bm{H}_0(\tcurl;\Omega),
$
where $\bm{\mathcal{I}}^1_{\rm c}$ denotes the canonical N\'ed\'elec interpolation operator of the first family. Then, invoking the interpolation estimate from \cite[Theorem~16.10]{Ern_Guermond_FEs_I_2021}, we obtain
\[
\Big(
\|\bm{v}-\bm{v}_h\|_{K}
+
h_K|\bm{v}-\bm{v}_h|_{\bm{H}^1(K)}
+
h_K\|\curl(\bm{v}-\bm{v}_h)\|_{K}
+
h_K^2|\curl(\bm{v}-\bm{v}_h)|_{\bm{H}^1(K)}
\Big)
\leq
C h_K^2 |\bm{v}|_{\bm{H}^2(K)}.
\]
Substituting this estimate into the  bound yields \eqref{eq:Nedelec-hp}.
\end{proof}

\section{\emph{H}$(\tdiv)$-conforming discontinuous Galerkin method}\label{sec: H(div)-DG}
In this section, we present the main ideas underlying the $\bm{H}(\tdiv)$-conforming dG method and state some useful results for the forthcoming analysis.

First, we introduce the discrete spaces based on the Raviart-Thomas (RT) finite element space. For $p\geq 1$, we have
\begin{align*}
	\bm{V}_{h}^p:=\bm{RT}_p(\mesh) \cap \bm{H}(\tdiv;\Omega).\end{align*}
In addition, the subspace with strongly enforced homogeneous boundary conditions and mean-free piecewise polynomial spaces are defined as:
\begin{equation}
	\bm{V}_{h0}^p:=\bm{RT}_p(\mesh) \cap \bm{H}_0(\tdiv;\Omega),\quad
	Q_h^p:=\{q\in \mathbb{P}_p(\mesh); (q,1)_{\Omega} =0\}.
\end{equation}

The discrete formulation reads: Find $(\bm{u}_h,p_h)\in \bm{V}_{h0}^p{\times} Q_h^p$ such that
\begin{alignat}{2}
	a_h(\bm{u}_h,\bm{v}_h) +b(\bm{v}_h,p_h)&=(\bm{f},\bm{v})_{\Omega}&&\quad \forall \bm{v}_h\in \bm{V}_{h0}^p,\label{eq:discrete1}\\
	-b(\bm{u}_h,q)&=0&&\quad \forall q\in Q_h^p\label{eq:discrete2},
\end{alignat}
with the discrete bilinear form $a_h(\bullet,\bullet)$ such that, for all $\bm{v}_h$, $\bm{w}_h \in \bm{V}_{h0}^p$,
\begin{align*}
    a_h(\bm{v}_h,\bm{w}_h):&=\sum_{K\in \mesh} (\mu \nabla \bm{v}_h, \nabla \bm{w}_h)_K
	-\sum_{F\in \Fall} (\mu \{\bm{n}{\times}(\nabla \bm{w}_h \n ) \} , \bm{n}{\times}\jump{\bm{v}_h})_F\\
	&\;-\sum_{F\in \Fall} (\mu \{(\bm{n}{\times}(\nabla \bm{v}_h  \n) ) \} , \bm{n}{\times}\jump{\bm{w}_h})_F+\sum_{F\in \Fall} ( \sigma \mu \bm{n}{\times}\jump{\bm{v}_h}, \bm{n}{\times}\jump{\bm{w}_h})_F,
\end{align*}
where the discontinuity penalty parameter is defined as:
\begin{equation}\label{def: penalty}
	\sigma|_F = \frac{C_{\sigma} p^2}{h_F} \qquad \forall F\in \Fall,
\end{equation}
with the user-dependent parameter $C_{\sigma} $ depending on the mesh-shape regularity, and should be chosen large enough \cite[(6.2)]{EggerWaluga2013}.

Next, we define the lifting operator $\mathcal{L}: \bm{V}_{h0}^{p+1}+\bm{V}\rightarrow [\bm{V}_{h0}^p]^d$, \begin{align}\label{def: lifting operator}
	(\mathcal{L}(\bm{v}),\underline{\bm{\theta}}_h )_{\Omega}=-\sum_{F\in \Fall} (\{\bm{n}{\times}(\underline{\bm{\theta}}_h\n)\}, \bm{n}{\times}\jump{\bm{v}})_F\quad \forall \underline{\bm{\theta}}_h\in \left[\bm{V}_{h0}^p\right]^d.
\end{align}
By using the discrete trace inequality \eqref{eq: discrete trace}, the following stability estimate holds
\begin{equation}
	\|\mathcal{L}(\bm{v})\|_{\Omega}^2\leq C\sum_{F\in \Fall}\left(\frac{p^2}{h_F}\right) \|  \bm{n}{\times} \jump{\bm{v}}\|_{F}^2 \label{eq:lifting}.
\end{equation}
Then the bilinear form $a_h(\bm{v},\bm{w})$ can be written as follows:
\begin{align}\label{eq:def-ah-lifting}
	a_h(\bm{v},\bm{w})
	=\sum_{K\in \mesh} (\mu \nabla \bm{v}, \nabla  \bm{w})_K +(\mu\mathcal{L}(\bm{w}), \nabla_h \bm{v})_{\Omega}
	+(\mu\mathcal{L}(\bm{v}),  \nabla_h \bm{w})_{\Omega}
	+\sum_{F\in \Fall} ( \sigma \mu \bm{n}{\times}\jump{\bm{v}},\bm{n}{\times} \jump{\bm{w}})_F.
\end{align}
For the simplification of the presentation, we define
\begin{equation}
	\begin{split}
		B_h((\bm{u}_h,p_h);(\bm{v}_h,q_h)):=a_h(\bm{u}_h,\bm{v}_h)+b(\bm{v}_h,p_h)-b(\bm{u}_h,q_h).
	\end{split}
	\label{eq:Bh}
\end{equation}
Then, we can recast the discrete formulation as follows: Find $(\bm{u}_h,p_h)\in \bm{V}_{h0}^p{\times} Q_h^p$ such that
\begin{align}
	B_h((\bm{u}_h,p_h);(\bm{v}_h,q_h))=(\bm{f},\bm{v}_h)_{\Omega},\quad \forall (\bm{v}_h,q_h)\in \bm{V}_{h0}^p{\times} Q_h^p.\label{eq:discrete}
\end{align}
Moreover, the above discrete scheme can be reformulated as: Find $\bm{u}_h\in \bm{V}_{h0}^p\cap \bm{H}(\tdiv^0;\Omega)$ such that for any $\bm{v}_h\in  \bm{V}_{h0}^p\cap \bm{H}(\tdiv^0;\Omega)$
\begin{equation}
	a_h(\bm{u}_h, \bm{v}_h) = (\bm{f},\bm{v}_h)_{\Omega}. \label{eq:discrete div-free}
\end{equation}
As shown in \cite{LedererSchoberl18}, the discrete problem~\eqref{eq:discrete} is well-posed.

\section{Local Helmholtz decomposition}\label{sec: Local Helmholtz Decomposition}

We derive a local Helmholtz decomposition of tensor-valued fields in $\underline{\bm{L}}^2(\oma)$. We only consider the case $d=3$, which can be viewed as an extension of \cite[Lemma 3.2]{DariDuranPadra1995} from $d=2$ to $d=3$. A key contribution of this work is the derivation of a stability estimate for the local Helmholtz decomposition on vertex patches, with a stability constant depending solely on the mesh shape-regularity.

Let $\varepsilon_{ijk}$ denote the Levi-Civita symbol for all $i,j,k\in\{1{:}d\}$.
We define the (vector-valued) curl of $\bm{v}$ has components $(\nabla {\times}\bm{v})_{i} := \varepsilon_{ijk} \partial_j v_{k}$ for all $i\in\{1{:}d\}$. Here and in what follows, we employ the usual summation convention on repeated indices. The (tensor-valued) gradient of a vector-valued field \(\bm v\) is defined componentwise by $
(\nabla \bm v)_{ij}:=\partial_j v_i$ for all $ i,j\in\{1{:}d\}$. For a tensor-valued field $\underline{\bm{A}} = (A_{ij})_{i,j\in\{1{:}d\}}$, its (tensor-valued) row-wise curl is defined as $(\curlrw{\underline{\bm{A}}})_{ij}:= \varepsilon_{jkl} \partial_k A_{il}$ for all $i,j\in\{1{:}d\}$, and its (vector-valued) row-wise divergence is defined as $(\divrw{\underline{\bm{A}}})_{i} = \sum_{j} \partial_j A_{ij}$ for all $i\in\{1{:}d\}$.

\begin{lemma}\label{lemma:helmholtz}(Local Helmholtz decomposition).
	For $\underline{\bm{\theta}}\in \underline{\bm{L}}^2(\oma)$,	there exists $\bm{\phi}\in \bm{H}^1_{0}(\oma)\cap \bm{H}(\tdiv^0;\oma)$, $q\in L^2_0(\oma)$,
	and $\underline{\bm{\beta}}\in \underline{\bm{H}}^1(\oma)$, such that
	\begin{align}\label{eq: HD identity}
		\underline{\bm{\theta}}=  \nabla \bm{\phi}+ \curlrw{\underline{\bm{\beta}}} - q\underline{I}.
	\end{align}
	Moreover, it holds
	\begin{align}\label{eq: HD stability}
		\| \nabla  \bm{\phi}\|_{\oma}
		+ \|q\|_{\oma}
		+ 	| \underline{\bm{\beta}}|_{\underline{\bm{H}}^1(\oma)}\leq C\|\underline{\bm{\theta}} \|_{\oma}.
	\end{align}
\end{lemma}

\begin{proof}
We consider the following Stokes problem: Find $(\bm{\phi},q)\in \bm{H}^1_{0}(\oma)\times L^2_0(\oma)$ such that
\begin{alignat}{2}
		( \nabla \bm{\phi}, \nabla \bm{v})_{\oma}-(q,\nabla{\cdot} \bm{v})_{\oma}&=(\underline{\bm{\theta}},\nabla \bm{v})_{\oma} &&\quad \forall \bm{v}\in \bm{H}^1_{0}(\oma),\label{eq:stokes1}\\
		(\nabla {{\cdot} }\bm{\phi},\chi)_{\oma}&=0&&\quad \forall \chi \in L^2_0(\oma).
	\end{alignat}
By the inf-sup theory of \cite[Lemma~3.1]{Hannukainen12} (see also \cite{Bacuta16}),
\begin{align}
\|\nabla\bm{\phi}\|_{\oma}^2
+C_B\|q\|_{\oma}^2
\le
\left(\frac{\sqrt5-1}{2}\right)^2
\|\underline{\bm{\theta}}\|_{\oma}^2,
\label{eq:stokes-stability}
\end{align}
where $C_B$ denotes the inf-sup constant for the divergence operator on the vertex patch $\oma$.

To bound $C_B$, we first recall the geometric properties of vertex patches. For an interior vertex, $\oma$ is star-shaped with respect to a ball whose radius is comparable to the diameter of the patch. This result is established in \cite[Proposition~8.2]{chaumontfrelet:hal-05204325} for all dimensions $d\ge2$ by reduction to the two-dimensional case, where the corresponding result is classical; see \cite{LeePre:79}. For a boundary vertex, uniform star-shapedness with respect to a single ball may fail, for example near re-entrant corners. In this case, $\oma$ can be decomposed into a uniformly bounded chain of star-shaped subdomains formed by unions of simplices sharing a common face; see the proof of \cite[Lemma~4.3]{dong:hal-05213366}. Consequently, by the Bogovski\u{\i} theory (see the discussion preceding the theorem together with \cite[Corollary~28 and Theorem~35]{GuzmanAbner21}), the inf-sup constant $C_B$ depends only on the mesh shape-regularity.

Therefore,
\begin{align}
\|\nabla\bm{\phi}\|_{\oma}
+\|q\|_{\oma}
\le
C\|\underline{\bm{\theta}}\|_{\oma},
\label{eq:estimate-stokes}
\end{align}
where the constant $C$ depends only on the mesh shape-regularity.

	Next, a simple computation shows that  \eqref{eq:stokes1} is equivalent to:
	\begin{align}
		( \nabla \bm{\phi}, \nabla \bm{v})_{\oma}-(q\underline{I} ,\nabla \bm{v})_{\oma}&=( \underline{\bm{\theta}},\nabla\bm{v})_{\oma}\quad \forall \bm{v}\in \bm{H}^1_{0}(\oma).\label{eq:stokes1-non}
	\end{align}
	By setting $\underline{\bm{\xi}}:=   \underline{\bm{\theta}}-\nabla \bm{\phi}+q\underline{I}$, where $\underline{\bm{\xi}}:=(\bm{\xi}_1,\bm{\xi}_2,\bm{\xi}_3)^{\top}$,
	we have $\bm{0}= (\bm{\xi},\nabla \bm{v})_{\oma}$, which implies $ \bm{\xi}_i \in \bm{H}(\tdiv^0;\oma)$, for $i=1,2,3$.  Furthermore, as the vertex patch $\oma$ is simply connected, applying the \cite[Theorem 32]{GuzmanAbner21} for each row $\underline{\bm{\xi}}$,  there exists $\underline{\bm{\beta}}\in \underline{\bm{H}}^1(\oma)$ such that $\bm{\xi}_i= \tcurl\underline{\bm{\beta}}_i$ with $\underline{\bm{\beta}}:=(\bm{\beta}_1,\bm{\beta}_2,\bm{\beta}_3)^{\top}$. This leads to~\eqref{eq: HD identity}.
	Moreover, using the fact that $\|\nabla {\bm{\beta}}_i\|_{\oma} \leq C \|\curl  {\bm{\beta}}_i\|_{\oma}$, where $C>0$ depends on the mesh shape regularity, we have the following bound
	\begin{align}
		| \underline{\bm{\beta}}|_{\underline{\bm{H}}^1(\oma)}  \leq C \|  \underline{\bm{\theta}}- \nabla \bm{\phi} +q\underline{I}\|_{\oma}
		\leq C  \|\underline{\bm{\theta}} \|_{\oma}.
	\end{align}
    Moreover, the above
	bound combined with the triangle inequality and the bounds~\eqref{eq:estimate-stokes} leads to~\eqref{eq: HD stability}. This completes the proof.
\end{proof}


\section{Pressure robust $hp$-a posteriori error analysis}\label{sec:apost}

In this section, we carry out the pressure-robust residual-based $hp$–a posteriori error analysis for the $ \bm{H}(\tdiv)$-conforming dG methods for the Stokes problem in three dimensions. Moreover, we assume $\bm{f}\in \bm{H}(\tcurl; \Omega)$ and the domain $\Omega$ is simply connected.

The goal is to establish an upper error bound on the approximation error
\begin{align}
	\bm{e}:=	\bm{u}-\bm{u}_h,\label{eq:error-decomposition}
\end{align}
which measures the difference between the exact solution of velocity~\eqref{eq:weak-B}, $\bm{u}$, and the $ \bm{H}(\tdiv)$-conforming dG solution of~\eqref{eq:discrete}, $\bm{u}_{h}$.
We will use the following local error indicators: For all $K\in \mesh$,
	\begin{align} \label{Error indicators}
& \qquad	\eta_{K,{\rm res}}^2
		:={} \mu^{-1}  \left( \frac{h_K}{p}\right)^4\| \bm{\Pi}_K^{p-2} (\curl{\bm{f}}) + \mu \curl{\Delta \bm{u}_h}\|_{K}^2,\\
		\eta_{K,{\rm sta}}^2
		:={}
		 & \mu\sum_{F\in \dK} \alpha_F\left(\frac{p^2}{h_F}\right)\|\bm{n}{\times} \jump{\bm{u}_h}\|_{F}^2,
\quad
\eta_{K, {\rm tan}}^2
		:={}
		\mu \sum_{F\in \dK} \alpha_F \left( \frac{h_F}{p}\right) \|\n {\times}\jump{\nabla_h \bm{u}_h} \|_{F}^2, \nonumber \\
		\eta_{K, {\rm nor}}^2
		:={}&
		\frac{\mu}{2} \!\sum_{F\in \dK \cap \Fint}  \left( \frac{h_F}{p}\right) \| \jump{ \nabla_h \bm{u}_h\bm{n}} \|_{F}^2,
\quad
		\eta_{K, {\rm tan},\Delta}^2
	:={}	\frac{\mu}{2} \!\sum_{F\in \dK \cap \Fint} \left( \frac{h_F}{p}\right)^3 \|\bm{n}{\times}\jump{ \Delta_h \bm{u}_h }  \|_{F}^2, \nonumber
		\end{align}
with $\alpha_F = 1$ for $F\in \Fb$ and $\alpha_F = \frac12$ otherwise. In addition, we  define the data oscillation term
	\begin{equation}
		\mathcal{O}(\bm{f})_K^2
		:= \mu^{-1}\left( \frac{h_K}{p}\right)^4\| \curl{\bm{f}} - \bm{\Pi}_K^{p-2} (\curl{\bm{f}}) \|_{K}^2.
		\label{est:oscillation}
	\end{equation}
\begin{remark}[Data oscillation term] In the definition of $\mathcal{O}(f)_K$, we employ $\bm{\Pi}_K^{p-2}$, which has a higher convergence rate compared to the broken $H^1$-norm error for smooth data. However, this choice is flexible; one may also consider applying the RT-projection $\bm{I}_{\rm RT}^{p-3} (\curl{f})$. As the convergence rate of such a choice is the same as the broken $H^1$-norm error for smooth enough solutions, and $(\curl{f} - \bm{I}_{\rm RT}^{p-3} (\curl{f}))$ is divergence-free.
\end{remark}

\subsection{Abstract error bound}
Our first step is to bound the velocity error by the sum of the dual norm of a suitable residual functional on the energy space $\bm{H}^1_0(\Omega)\cap \bm{H}(\tdiv^0; \Omega)$ and a nonconforming error term measuring the deviation of $\bm{u}_{h}$ from this space.
\begin{lemma}[Abstract error bound]\label{Lemma: abstract error bound}
	The following holds:
	\begin{equation}\label{eq: abstract error bound}
		\|\mu^{\frac12}\nabla_h \bm{e}\|_{\Omega}^2
		\leq  \mu^{-1}\|\mathcal{R}_{\mesh}\|_{\bm{H}^{-1}(\Omega)}^2
		+\inf_{\bm{w}\in \bm{H}^1_0(\Omega)\cap \bm{H}(\tdiv^0;\Omega)}\|\mu^{\frac12} \nabla_h (\bm{w} - \bm{u}_{h}) \|_{\Omega}^2,
	\end{equation}
	where $\mathcal{R}_{\mesh}\in \bm{H}^{-1}(\Omega)$ denotes the residual functional such that
	\begin{equation}\label{def: residual functional}
		\mathcal{R}_{\mesh}(\bm{w}) := (\mu \nabla_h \bm{e}, \nabla \bm{w})_{\Omega}
		= (\bm{f},\bm{w})_\Omega - (\mu \nabla_h  \bm{u}_{h}, \nabla \bm{w})_{\Omega}, \qquad \forall \bm{w}\in \bm{H}^1_0(\Omega)\cap \bm{H}(\tdiv^0; \Omega),
	\end{equation}
	and its dual norm is defined as
	\begin{equation} \label{eq:def_res_dual}
		\|\mathcal{R}_{\mesh}\|_{\bm{H}^{-1}(\Omega)} := \sup_{\bm{0}\neq\bm{w}\in \bm{H}^1_0(\Omega)\cap \bm{H}(\tdiv^0; \Omega)} \frac{\mathcal{R}_{\mesh}(\bm{w})}{\|\nabla \bm{w}\|_\Omega}.
	\end{equation}
\end{lemma}
\begin{proof}
	Choosing $\bm{v}:=\argmin_{\bm{w}\in \bm{H}^1_0(\Omega)\cap \bm{H}(\tdiv^0; \Omega)} \|\mu^{\frac12}\nabla_{h} (\bm{w} - \bm{u}_{h}) \|_{\Omega}$, we obtain the orthogonality relation
	\begin{equation*}
		(\mu \nabla_{h} (\bm{v} - \bm{u}_{h}), \nabla  \bm{w})_{\Omega} = 0 \qquad \forall \bm{w}\in \bm{H}^1_0(\Omega)\cap \bm{H}(\tdiv^0; \Omega).
	\end{equation*}
	The Pythagoras identity then gives
	\begin{equation*}
		\|\mu^{\frac12} \nabla_{h} \bm{e}\|_{\Omega}^2 =  \| \mu^{\frac12} \nabla (\bm{u}-\bm{v})\|_{\Omega}^2 + \|\mu^{\frac12} \nabla_{h}  (\bm{v}-\bm{u}_{h})\|_{\Omega}^2.
	\end{equation*}
	For the first term on the right-hand side, using again the above orthogonality relation and the definition of the dual norm, we infer that
	\begin{equation*}
		\|\mu^{\frac12} \nabla (\bm{u}-\bm{v})\|_{\Omega}^2  = (\mu \nabla_{h} \bm{e}, \nabla (\bm{u}-\bm{v}))_{\Omega} = \mathcal{R}_{\mesh}(\bm{u}-\bm{v})
		\leq  \mu^{-\frac12}\|\mathcal{R}_{\mesh}\|_{\bm{H}^{-1}(\Omega)} \|\mu^{\frac12} \nabla (\bm{u}-\bm{v})\|_{\Omega}.
	\end{equation*}
	Putting everything together completes the proof.
\end{proof}

\subsection{Bound on dual residual norm}

To derive a pressure-robust upper bound for the dual residual norm, we invoke the interpolation operator \eqref{lemma: Global BS} and employ the generalized Bogovskiĭ operator from \cite[Section 1.2]{Ahmed21}; see also \cite[Theorem~4.9]{CostabelMcIntosh2010}. Let $\bm{v} \in \bm{H}^1_{0}(\Omega) \cap \bm{H}(\tdiv^0;\Omega)$, and assume that $\Omega$ is a simply connected Lipschitz domain. Then there exists $\bm{\zeta}\in \bm{H}^2_0(\Omega)$ such that
\begin{equation}\label{eq: Bogovskii operator}
	\bm{v} = \curl{ \bm{\zeta}}.
\end{equation}
Moreover, the following stability estimate holds:
\begin{align}
	\|\bm{\zeta}\|_{\bm{H}^2(\Omega)}\leq C(\Omega) \| \nabla \bm{v}\|_{\Omega}\label{eq:stability-H2}.
\end{align}

\begin{lemma}[Bound on dual residual norm]\label{lemma: dual norm of residual} The following holds:
	\begin{equation} \label{eq:dual_res1}
\begin{split}
		\mu^{-1}\|\mathcal{R}_{\mesh}\|_{\bm{H}^{-1}(\Omega)}^2
\leq& C(\Omega) \min \left\{ \sum_{K\in \mesh}
		  \left\{
		p \eta_{K,{\rm sta}}^2	+ \eta_{K,{\rm nor}}^2+\eta_{K,{\rm tan},\Delta}^2
		 + \eta_{K,{\rm res}}^2+ \mathcal{O}(\bm{f})^2_K\right\}, \right.\\
& \qquad \qquad  \left.
 \sum_{K\in \mesh}
p^2\left\{
 \eta_{K,{\rm sta}}^2 + \eta_{K,{\rm nor}}^2+ \mathcal{O}(\bm{f})^2_K
\right\}
\right\}.
\end{split}
	\end{equation}
\end{lemma}
\begin{proof}
	The proof is postponed to Section \ref{sec: proof dual norm of residual}.
\end{proof}
\subsection{Bound on nonconforming error}
To bound the nonconforming error, it suffices to pick any function in $\bm{H}^1_{0}(\Omega) \cap \bm{H}(\tdiv^0;\Omega)$. In this section, we show how to reconstruct from $\bm{u}_{h}\in \bm{V}_{h0}^p\cap \bm{H}(\tdiv^0;\Omega)$ to such a function  using  hat basis functions of partition of unity and local solves on vertex patches.
Recall that the  hat basis functions satisfy the following partition-of-unity property:
\begin{equation}\label{eq: PU}
	\sum_{\ba\in \vertice} \psi_{\ba} =1 \quad \text{on } \Omega.
\end{equation}

\begin{definition}[Patchwise and global Stokes-reconstruction]\label{def: Stokes reconstruction}
	Let $\bm{w}_h \in  \bm{V}_{h0}^p$. For all $\ba\in \vertice$, let $\bm{w}_{\ba} \in \bm{H}^1_0{(\oma)}$ and $p_{\ba} \in L^2_0(\oma)$ solve the following well-posed problem:
	\begin{alignat}{2}
		(\nabla \bm{w}_{\ba}, \nabla\bm{v}_{\ba})_{\oma}  -(p_{\ba},\nabla{\cdot} \bm{v}_{\ba})_{\oma} &=	(\nabla_h(\psi_{\ba}\bm{w}_h) + \mathcal{L}(\psi_{\ba}\bm{w}_h), \nabla\bm{v}_{\ba})_{\oma}
		&&\quad \forall \bm{v}_{\ba}\in \bm{H}^1_{0}(\oma), \label{eq:reconstruction-Stokes 1}\\
		(q_{\ba},\nabla{\cdot} \bm{w}_{\ba})_{\oma}	&= (q_{\ba},\nabla{\cdot} (\psi_{\ba} {\bm{w}_h}))_{\oma}  &&\quad \forall q_{\ba}\in  L^2_0 (\oma)\label{eq:reconstruction-Stokes 2}.
	\end{alignat}
	Extending $\bm{w}_{\ba}$ by zero to $\Omega$,  set
	\begin{equation}\label{def: Stokes reconstruction global}
		\mathcal{R}_s(\bm{w_h})  :=  \sum_{\ba \in \vertice} \bm{w}_{\ba}.
	\end{equation}
\end{definition}
An application of \eqref{eq:reconstruction-Stokes 1}  implies
\begin{align}
	(\nabla_h(\bm{w}_{\ba}-\psi_{\ba}\bm{w}_h), \nabla\bm{v}_{\ba})_{\oma}= (\mathcal{L}(\psi_{\ba}\bm{w}_h),\nabla \bm{v}_{\ba})_{\oma}\quad \forall \bm{v}_{\ba} \in \bm{H}^1_{0}(\oma)\cap \bm{H}(\tdiv^0;\oma).\label{eq:difference-lifting}
\end{align}

\begin{remark}[local problem]
	As the vertex patch $\oma$ is a simply connected
	Lipschitz domain with connected boundary,  the function $\bm{w}_h \in  \bm{V}_{h0}^p$, and $\psi_{\ba}\in H^1(\oma)$, we infer
	$
	(\nabla{\cdot}(\psi_{\ba} \bm{w}_h),1)_{\oma} = (\n {\cdot}(\psi_{\ba} \bm{w}_h),1)_{\partial \oma} =0,
	$
	which implies that $\nabla{\cdot}(\psi_{\ba} \bm{w}_h)\in L^2_0 (\oma)$. Therefore, the above problem \eqref{eq:reconstruction-Stokes 1}--\eqref{eq:reconstruction-Stokes 2} is well-posed. Moreover, an equivalent definition is
	\begin{equation}\label{def: argmin of Stokes reconstruction}
		\bm{w}_{\ba}:=\argmin_{
			\substack
			{\bm{\rho}_{\ba} \in \bm{H}^1_0(\oma),} \\
			\substack{\nabla {\cdot} \bm{\rho}_{\ba} =  \nabla {\cdot} (\psi_{\ba} \bm{w}_h)}
		} \| (\nabla_h (\psi_{\ba} \bm{w}_h)+  \mathcal{L}(\psi_{\ba} \bm{w}_h))-\nabla \bm{\rho}_{\ba}\|_{\oma}.
	\end{equation}
\end{remark}
A simple but crucial observation is that $\bm{w}_{\ba} \in \bm{H}^1_0(\oma)$ with $\nabla {\cdot} \bm{w}_{\ba}=\nabla {\cdot} (\psi_{\ba} \bm{w}_h)$. In addition, using the partition-of-unity \eqref{eq: PU},
we have
\begin{equation}\label{divergence of Stokes_reconstruction}
	\nabla {\cdot} 	\mathcal{R}_s(\bm{w_h}) = \sum_{\ba\in \vertice} (\nabla {\cdot} \bm{w}_{\ba})
	= \sum_{\ba\in \vertice} (\nabla {\cdot} (\psi_{\ba} \bm{w}_h))
	= \nabla {\cdot} \left( \sum_{\ba\in \vertice} (\psi_{\ba} \bm{w}_h) \right)
	=   \nabla {\cdot} \bm{w}_h,
\end{equation}
which implies
\begin{equation}\label{def: global Stokes reconstruction}
		\mathcal{R}_s(\bm{w_h}) \in \bm{H}^1_0(\Omega), \qquad \nabla {\cdot} \mathcal{R}_s(\bm{w_h}) = \nabla {\cdot} \bm{w}_h.
\end{equation}


\begin{lemma}[Bound on nonconforming error]\label{lemma: nonconforming error total}
	Since the $\bm{H}(\tdiv)$-dG solution satisfies $\bm{u}_{h}\in \bm{V}_{h0}^p\cap \bm{H}(\tdiv^0;\Omega)$, there exists a constant $C>0$, depending only on the mesh shape-regularity, such that
	\begin{subequations} \label{eq:main_bounds} \begin{align}
			\|\mu^{\frac12}\nabla_h (	\mathcal{R}_s(\bm{u_h}) -\bm{u}_h) \|_{\Omega}^2
			\leq {} C \sum_{K\in \mathcal{T}_h} \left\{
			\eta_{K,{\rm tan}}^2 + \eta_{K,{\rm sta}}^2
			\right\},
			\label{eq: Global nonconforming error bound weighted broken curl}
		\end{align}
	\end{subequations}
\end{lemma}
\begin{proof}
	The proof is postponed to Section \ref{sec: proof full bound}.
\end{proof}

\begin{remark}[Inhomogeneous boundary condition]
The above Stokes-reconstruction can be modified to account for inhomogeneous boundary conditions of the form $\bm{u} = \bm{g}_{\rm D}$ on $\partial \Omega$.
For every mesh vertex $\ba \in \vertice $, we then solve for $\bm{w}_{\ba} \in \bm{H}^1_{g} (\oma):=\{\bm{v}\in \bm{H}^1(\oma) \;|\; \bm{v}|_{\partial \oma \cap \partial \Omega} = \psi_{\ba}\bm{g}_{\rm D}, \; \bm{v}|_{\partial \oma \cap \Omega} =  \bm{0} \}$
\begin{alignat}{2}
		(\nabla \bm{w}_{\ba}, \nabla\bm{v}_{\ba})_{\oma}  -(p_{\ba},\nabla{\cdot} \bm{v}_{\ba})_{\oma} &=	(\nabla_h(\psi_{\ba}\bm{w}_h) + \mathcal{L}(\psi_{\ba}\bm{w}_h), \nabla\bm{v}_{\ba})_{\oma}
		&&\quad \forall \bm{v}_{\ba}\in \bm{H}^1_{0}(\oma), \label{eq:reconstruction-Stokes 1-inhomo}  \nonumber \\
		(q_{\ba},\nabla{\cdot} \bm{w}_{\ba})_{\oma}	&= (q_{\ba},\nabla{\cdot} (\psi_{\ba} \bm{w}_h))_{\oma}  &&\quad \forall q_{\ba}\in  L^2_0 (\oma). \nonumber
	\end{alignat}
Then, extending $\bm{w}_{\ba}$ by zero to $\Omega$, $\mathcal{R}_s(\bm{w_h}) \in \bm{H}^1(\Omega)\cap \bm{H}(\tdiv^0;\Omega)$ is still defined as in~\eqref{def: Stokes reconstruction global}, and the partition-of-unity property~\eqref{eq: PU} implies that $\mathcal{R}_s(\bm{u_h})$ also satisfies the inhomogeneous boundary conditions. Finally, the upper bound from Lemma~\ref{lemma: nonconforming error total} still holds provided the tangential jump $\eta_{K, {\rm tan}}$  and penalty jump $\eta_{K,{\rm sta}}$ are redefined as follows: For all $F\in\Fb$,
$$
\n {\times}\jump{ \nabla_h \bm{u}_h} |_F :=
\n {\times} \nabla (\bm{u}_h)|_F  - \n {\times} \nabla (\bm{g}_{\rm D})|_F,
\qquad
\n {\times}\jump{ \bm{u}_h}|_F :=
\n {\times}  (\bm{u}_h)|_F  - \n {\times}  (\bm{g}_{\rm D})|_F.
$$
\end{remark}

\subsection{Main result}

We are now ready to state our main result which contains a global upper bound and a local lower bound.

\begin{theorem}[$hp$-upper error bounds]
	\label{theorem: a posterioir error bound}
	The following holds:
	\begin{align*}
		\|\mu^{\frac12}\nabla_h \bm{e}\|_{\Omega}^2
		+\mu \sum_{F\in \Fall} \left(\frac{p^2}{h_F}\right)\|\bm{n}{\times} \jump{\bm{u}_h}\|_{F}^2
		& \leq C(\Omega)\eta^2,
	\end{align*}
with the estimator $\eta$ defined as:
\begin{equation}\label{def: estimator}
\begin{split}
\eta^2 : =    \sum_{K\in \mesh}
\left\{\eta_{K,{\rm tan}}^2
\right\}
+ &\min \left\{ \sum_{K\in \mesh}
		  \left\{
		p \eta_{K,{\rm sta}}^2	+ \eta_{K,{\rm nor}}^2+\eta_{K,{\rm tan},\Delta}^2
		 + \eta_{K,{\rm res}}^2+ \mathcal{O}(\bm{f})^2_K\right\}, \right.\\
& \qquad  \left.
 \sum_{K\in \mesh}
p^2\left\{
 \eta_{K,{\rm sta}}^2 + \eta_{K,{\rm nor}}^2+ \mathcal{O}(\bm{f})^2_K
\right\}
\right\},
\end{split}
\end{equation}
	where the constant $C(\Omega)>0$ depends on the mesh shape-regularity and on the geometry of the domain.
\end{theorem}
\begin{proof}
	Combine Lemmas~\ref{Lemma: abstract error bound}, \ref{lemma: dual norm of residual},
	and \ref{lemma: nonconforming error total}.
\end{proof}

\begin{remark}[$\bm{BDM}_p$ spaces] \label{rem:BDM}
The \emph{a posteriori} error analysis can also be carried out for the $\bm{BDM}_p$ space in both two and three dimensions with only minor modifications. The nonconforming error analysis, as well as the conforming error analysis based on the modified global Babu\v{s}ka--Suri operator, is identical for the $\bm{BDM}_p$ and $\bm{RT}_p$ spaces. For the conforming error analysis employing the $hp$-version Nédélec projection-based interpolation operator, the required approximation results are, to the best of the authors' knowledge, currently available only for the first-family Nédélec spaces and the $\bm{RT}_p$ spaces (see, for example, \cite[Section 2.3]{Boffibook}). Consequently, for the $\bm{BDM}_p$ spaces, the conforming \emph{a posteriori} error bound can presently be established only in the $h$-version setting.
\end{remark}

\begin{theorem}[Local lower error bound]\label{theorem: lower bound}
	We will prove the following result: For all $K\in \mathcal{T}_h$,
	\begin{align}
	\eta_{K,{\rm res}}&\leq C(p)\Big( \|\mu^{\frac12}\nabla \bm{e}\|_{K}+ \mathcal{O}(\bm{f})_K\Big)\label{eq:error-cell-residual},\\
	\eta_{K,{\rm nor}}&\leq C(p) \sum_{K\in \omega_K}\Big(\|\mu^{\frac12}\nabla\bm{e}\|_{K}+\eta_{K,{\rm sta}}+\mathcal{O}(\bm{f})_K\Big) ,\label{eq:error-normal-flux}\\
	\eta_{K,{\rm tan},\Delta}&\leq C(p)\sum_{K\in \omega_K}\Big( \|\mu^{\frac12} \nabla\bm{e}\|_{K}  +\mathcal{O}(\bm{f})_K \Big),\label{eq:error-delta-uh}\\
	\eta_{K,{\rm tan}}&\leq  C(p)	\eta_{K,{\rm sta}},\label{eq:error-tangential}
\end{align}
where $\omega_K$ denotes the union of elements that share the face with $K$.
\end{theorem}
\begin{proof}
	The proof is postponed to Section \ref{sec: proof lower bound}.
\end{proof}

\begin{remark}[$p$-dependence of the local lower error bound]
We emphasize that we do not track the explicit $p$-dependence of the localized lower error bounds. Indeed, these bounds follow from straightforward adaptations of the arguments developed in~\cite{Dong21} for dG methods. Moreover, they exhibit a pessimistic algebraic $p$-suboptimality compared to the numerical observations, since their proof relies on $C^1$ bubble functions and $H^2$ extension operators, for which the $hp$-techniques developed in~\cite{Melenk01} for $H^1$ functions are not applicable. Deriving sharp $p$-explicit local lower error bounds remains an open problem and is left for future work.
\end{remark}

\section{Numerical experiments}\label{sec:Numerical example}
In this section, we illustrate the performance of the proposed scheme on various tests in both 2D and 3D, including smooth and nonsmooth solutions. An adaptive scheme is also presented and tested in this section, in both two and three dimensions.

The numerical experiments are conducted with the
\texttt{ngsolve} library \cite{Schoberl14}. Unless otherwise specified, we compute the convergence rate by taking the average of the last three points in each figure. In this section,  we consider ``error'' and ``estimator'' defined by
$
    \texttt{error}=\{\|\mu^{\frac12}\nabla_h(\bm{u}-\bm{u}_h)\|_{\Omega}^2
+\mu
\sum_{F\in \Fall} \left(\frac{p^2}{h_F}\right)\|\bm{n}{\times} \jump{\bm{u}_h}\|_{F}^2\}^{\frac12},
$
and $  \texttt{estimator}: = \eta$.
In addition, in order to illustrate the efficiency of the error estimators, we consider the following effectivity index defined by the ratio between the estimator and error:
\begin{equation}
    \texttt{effectivity index} = \frac{\texttt{estimator}}{\texttt{error}}.
\end{equation}

Numerical tests are conducted on both BDM and RT basis, but we only present the results with RT basis here, since our numerical analysis focuses mainly on RT basis, and results for BDM basis is eventually similar to those of RT basis.

\subsection{Smooth test: uniform refinements}
We consider the analytic solution in the 2D domain $(0,1)^2$ given by
\begin{equation}\label{eq_smooth_test_2d}
    \psi(x,y)=(1-\cos(2\pi x))(1-\cos(2\pi y)),\quad \bm{u}(x,y)=(\partial_y \psi, -\partial_x \psi)^{\top},\quad p(x,y) = \sin(2\pi x)\cos(2\pi y).
\end{equation}

In Figure \ref{fig:RT_2D}, we illustrate the convergence curve for various $h$, $p$ and $\mu$ for test case \eqref{eq_smooth_test_2d}, for errors and estimators. More precisely, in Figures \ref{fig:RT_2D}, $h$-refinement tests are presented in the left and middle panels, for $p\in \{1,2,3,4,5,6\}$ with $\mu=1$ and $\mu\in \{10^4,10^2,10^{-2},10^{-4}\}$ with $p=2$, respectively; $p$-refinement is consider in the right panels, where errors and estimators for $h=2{\times}10^{-1}$ and $\mu=1$ are presented.




\begin{figure}
    \centering
    \includegraphics[width=0.3\linewidth]{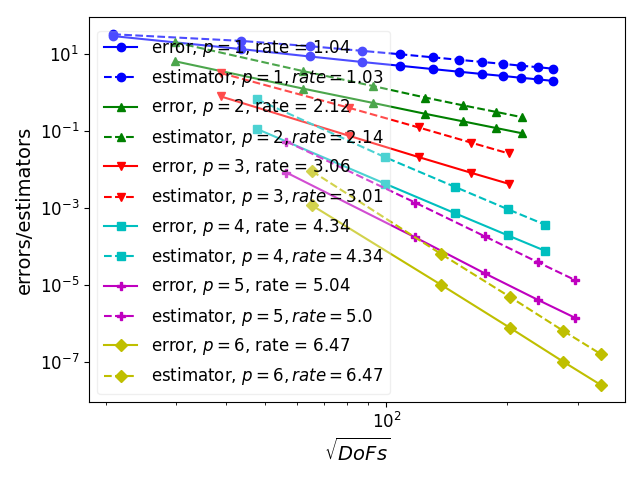}
    \includegraphics[width=0.3\linewidth]{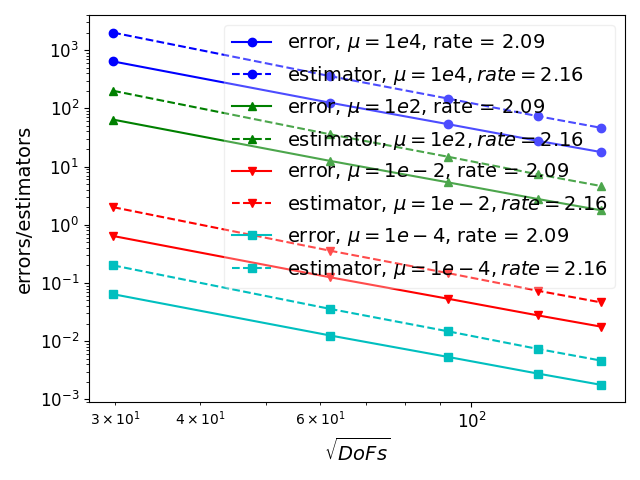}
    \includegraphics[width=0.3\linewidth]{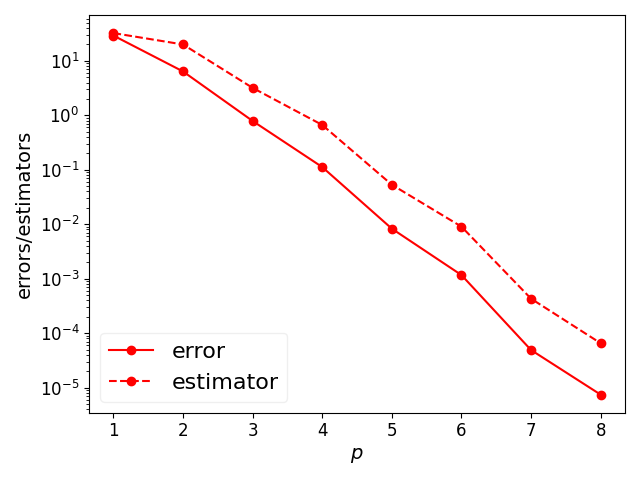}

    \includegraphics[width=0.3\linewidth]{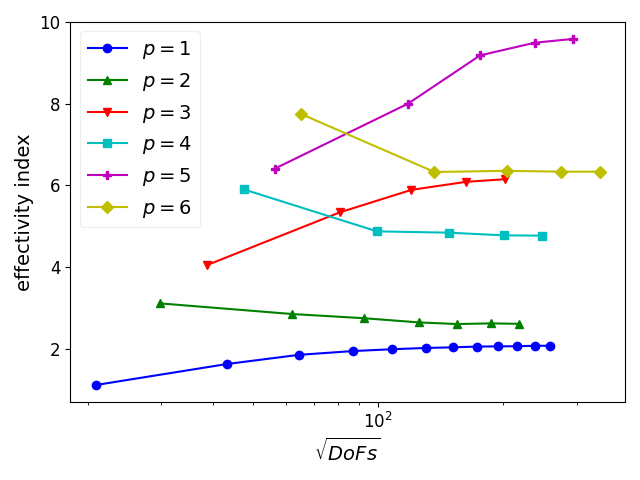}
    \includegraphics[width=0.3\linewidth]{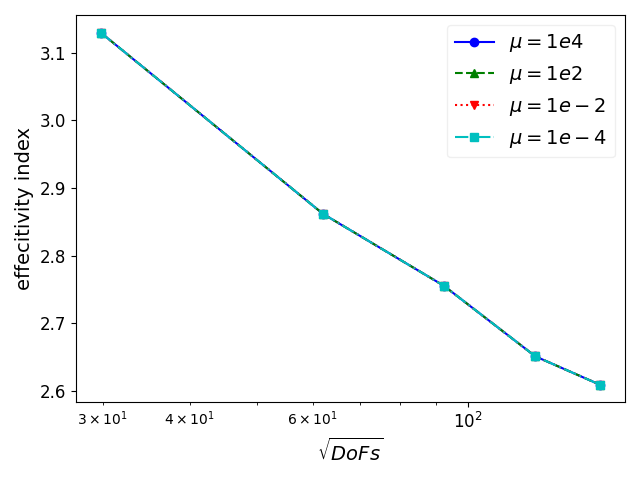}
    \includegraphics[width=0.3\linewidth]{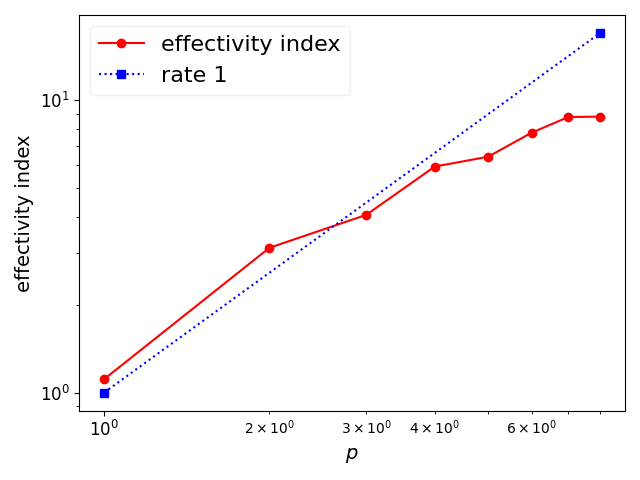}
    \caption{Accuracy test on \eqref{eq_smooth_test_2d}: left: $\mu=1$; middle: $p=2$; right: $\mu=1$, $h=2{\times}10^{-1}$.}
    \label{fig:RT_2D}
\end{figure}

We have the following observations from Figure \ref{fig:RT_2D}: 1. The optimal convergence rate $\mathcal{O}(\mathrm{DoFs}^{-\frac{p}{2}})$ is attained under $h$-refinement.  2. For each fixed polynomial degree $p$, the effectivity index remains asymptotically bounded as $h\to0$. Interestingly, the effectivity indices for odd and even polynomial degrees exhibit different asymptotic behaviours.  In addition, similar behaviour is observed for other manufactured smooth solutions. 3. The pressure-robustness of both the error and the estimator is clearly observed in all discrete settings considered.  4. Under $p$-refinement, exponential convergence is observed for both of the error and the estimator. Moreover, the corresponding effectivity index grows like $\mathcal{O}(p^{\alpha})$ with $\alpha\approx1$.

\subsection{Adaptive scheme}

We consider here an adaptive algorithm
with the usual structure
\begin{equation}\label{eq_adaptive_scheme}
    \textbf{SOLVE}
    \qquad \Longrightarrow \qquad
    \textbf{ESTIMATE}
    \qquad \Longrightarrow \qquad
    \textbf{MARK}
    \qquad \Longrightarrow \qquad
    \textbf{REFINE}.
\end{equation}
This structure is iteratively applied until either a given tolerance for $\eta$ is achieved or a given maximum iteration number $N_{\rm{iter}}$ is attained.
The \textbf{ESTIMATE} step is driven by using the error estimator~$\eta$;
to this aim, we propose an algorithm based on the local error estimator $\eta_K$ on each mesh cell $K\in \mesh$. More precisely, for each cell $K\in\mesh$, we consider the local error estimator $\eta_K$ denoted the local error contribution of the global estimator $\eta$ in \eqref{def: estimator}  on cell  $K$ such that $\eta_K^2:=(\eta|_K)^2$.
As for the \textbf{MARK} step, we use D\"orfler’s marking
with a given threshold $\theta_{\rm{refine}}$ in~$(0,1]$.
The \textbf{REFINE} step is realized by the default refine strategy of the \texttt{ngsolve} library.

\subsection{Nonsmooth test: singular $L$-shaped domain in 2D}
To demonstrate the performance of the proposed scheme in the low regularity setting, we consider
a Stokes problem on the L-shaped domain $\Omega=(-1,1)^2 \backslash [0,1]{\times} [-1,0]$ with the exact solution in polar coordinates defined by
\begin{equation}\label{eq_nonsmooth_test_Lshape}
	\bm{u}(r, \theta)=\left(
	\begin{array}{r@{\;\;}l}
		r^{\lambda}((1+\lambda)\sin(\theta)\psi(\theta)+
		\cos(\theta)\psi'(\theta)) \\
		r^{\lambda}(-(1+\lambda)\cos(\theta)\psi(\theta)+
		\sin(\theta)\psi'(\theta))
	\end{array}
	\right), \quad p=\mu p_1+p_2,
\end{equation}
where $p_1(r, \theta)=- r^{\lambda-1}((1+\lambda)^2\psi'(\theta)
	+\psi'''(\theta))/(1-\lambda)$, $p_2=x^3+y^3$
and
\begin{eqnarray*}
	\psi(\theta)=\sin((1+\lambda)\theta)\cos(\lambda\omega)/(1+\lambda)
	-\cos((1+\lambda)\theta)
	-\sin((1-\lambda)\theta)\cos(\lambda\omega)/(1-\lambda)+
	\cos((1-\lambda)\theta),
\end{eqnarray*}
with $\lambda\approx0.54448373678246$ and $\omega=3\pi/2$, note that $-\Delta \bm{u}+\nabla p_1=0$, thus $-\mu\Delta \bm{u}+\nabla p=\bm{f}$, where $\bm{f}=\nabla p_2$. In addition, it holds $(\bm{u}, p)\in \bm{H}^{1+\lambda}(\Omega){\times} H^{\lambda}(\Omega)$. To illustrate the pressure-robust feature of the proposed scheme, we fix $\mu=10^{-3}$ in this test. We note that this test has a corner singularity for pressure at the corner $(0,0)$, which is expected to be captured by adaptive schemes.

In Figure \ref{fig:mesh_Lshape}, we illustrate the pressure-robust feature of the proposed estimators by considering two meshes produced through our adaptive algorithm \eqref{eq_adaptive_scheme} with our estimator \eqref{def: estimator} and a non pressure-robust estimator
\begin{align*}
    \tilde{\eta}^2 = \sum_{K\in \mesh}\left\{\mu^{-1}\left( \frac{h_K}{p}\right)^2\|\bm{f} + \mu \Delta \bm{u}_h - \nabla p_h\|_{K}^2 + \eta_{K,\rm{nor}}^2 + \eta_{K,\rm{sta}}^2 + \eta_{K,\rm{tan}}^2 \right\},
\end{align*}
respectively. The initial mesh is constructed by the default mesh generator of \texttt{ngsolve} with $h=0.4$ and $\theta_{\rm{refine}}=0.3$. Maximum iteration number is $N_{\rm{iter}}=5$.

\begin{figure}[htb]
    \centering
    \includegraphics[width=0.35\linewidth]{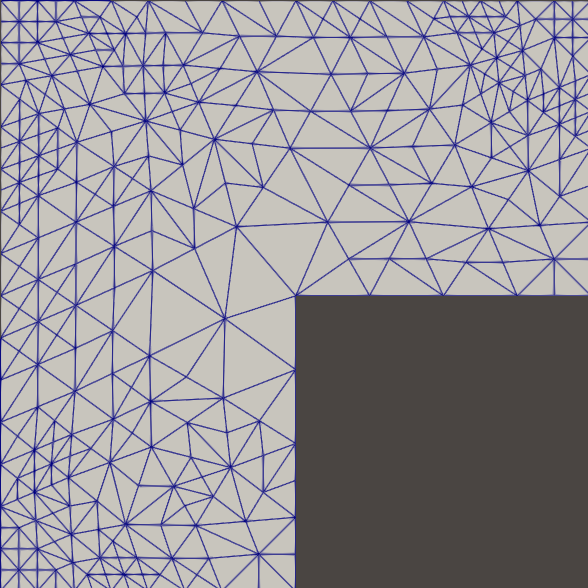}\qquad
    \includegraphics[width=0.35\linewidth]{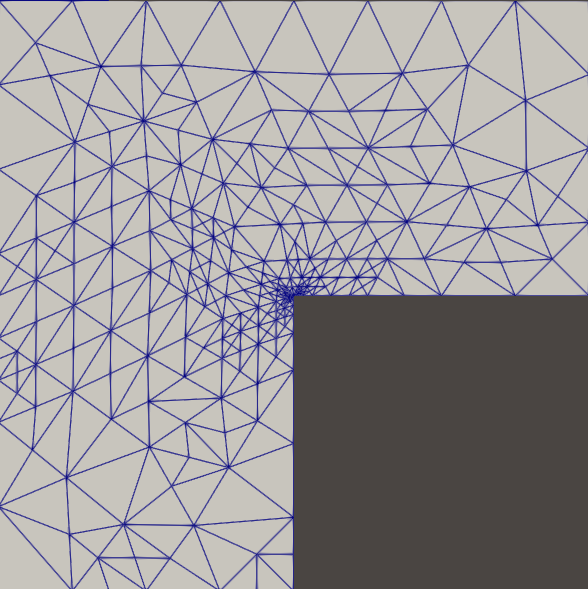}
    \caption{Meshes produced by our adaptive scheme \eqref{eq_adaptive_scheme} on test case \eqref{eq_nonsmooth_test_Lshape}, with $\theta_{\rm{refine}}=0.3$, $p=1$, $N_{\rm{iter}}=5$: left: non pressure-robust estimator; right: pressure-robust estimator \eqref{def: estimator}.}
    \label{fig:mesh_Lshape}
\end{figure}
We observe from Figure \ref{fig:mesh_Lshape} that the proposed estimator \eqref{def: estimator} successfully captures the corner singularity in pressure, whereas the non-pressure robust estimator  fails to refine near the corner singularity. This verifies the pressure-robust feature of our a posteriori analysis.

Then, in Figure \ref{fig:rate_Lshape}, we test the adaptive scheme with estimator \eqref{def: estimator} on the singular test case \eqref{eq_nonsmooth_test_Lshape}. The initial mesh is set as a quasi-uniform mesh with $h=0.4$, generated by the default mesh generator of \texttt{ngsolve}. Errors and estimators are presented in the figure; for comparison, errors measured on uniform refinement with $p=1$ are also illustrated.

\begin{figure}
    \centering
    \includegraphics[width=0.48\linewidth]{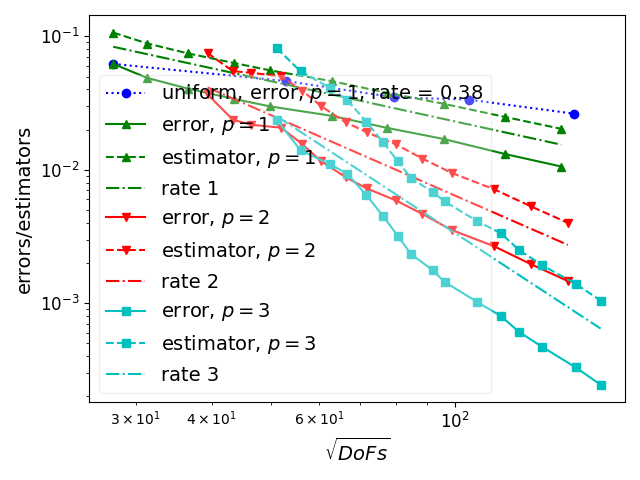}
    \caption{Convergence study for adaptive scheme \eqref{eq_adaptive_scheme} with estimator \eqref{def: estimator} on test case \eqref{eq_nonsmooth_test_Lshape}, with $\theta_{\rm{refine}}=0.3$, $p\in\{1,2,3\}$.}
    \label{fig:rate_Lshape}
\end{figure}

From Figure \ref{fig:rate_Lshape}, we have the following observations: 1. the uniform refinement leads to a slow convergence, whereas adaptive refinement has better convergence rate; 2. expected optimal convergence rates $\mathcal{O}({\rm{DoFs}}^{-\frac{p}{2}})$ are attained, in the entire range of $p\in\{1,2,3\}$. Therefore, the optimal convergence behavior is numerically justified on the $L$-shape domain singular test, for the adaptive algorithm associated with our error estimator \eqref{def: estimator}.

\subsection{Nonsmooth test: corner singularity in 3D}
We consider a corner singularity test on the unit cube $(0,1)^3$ here, with a singularity located at $(0,0,0)$ for pressure, as a 3D version of the $L$-shape domain example \eqref{eq_nonsmooth_test_Lshape}.
The exact solution is defined by
\begin{equation}\label{eq_nonsmooth_3D}
    \bm{u}(x,y,z) = A p_1 \left(
         x,
         y,
         z\right)^{\top} + B r^2 \nabla p_1, \quad p= \mu p_1 + p_2,
\end{equation}
where
\begin{align*}
    &p_1 = \rho^\lambda \cos(\lambda \theta),\quad p_2 =x^2+y^2+z^2,\quad \rho^2= s^2 +z^2,\quad r^2=x^2+y^2+z^2, \quad \theta = \text{atan2}(z,s),\\
    &s=\frac{x+y}{\sqrt{2}},\quad A=\frac{-\lambda}{(\lambda+1)(2\lambda+3)}, \quad B=\frac{\lambda+3}{2(\lambda+1)(2\lambda+3)},\quad \lambda=-\frac{4}{5},\quad \mu = 10^{-3}.
\end{align*}
We have the following remarks on the the exact solution constructed above: 1. it satisfies $(\bm{u},p)\in \bm{H}^{\frac{5}{2}+\lambda-\epsilon}(\Omega){\times} H^{\frac{3}{2}+\lambda-\epsilon}(\Omega)$ for arbitrarily small $\epsilon>0$; 2. $-\mu\Delta \bm{u} +\nabla p_1 = 0$; 3. $\mu$ is set small as in \eqref{eq_nonsmooth_test_Lshape} to illustrate the pressure-robust feature of our error estimator.

In Figure \ref{fig:rate_3D}, we present the numerical results for the corner singularity test \eqref{eq_nonsmooth_3D}, conducted by an adaptive algorithm with $\theta_{\rm{refine}}=0.2$, initial mesh generated by the default mesh generator of \texttt{ngsolve} with $h=0.2$.

\begin{figure}[htb]
    \centering
    \includegraphics[width=0.4\linewidth]{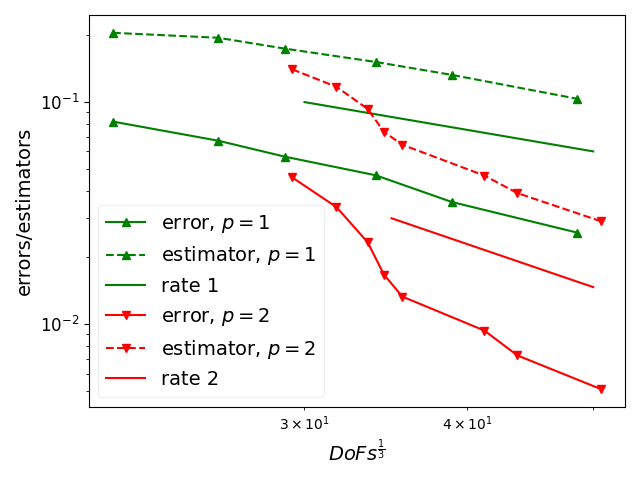}
    \includegraphics[width=0.28\linewidth]{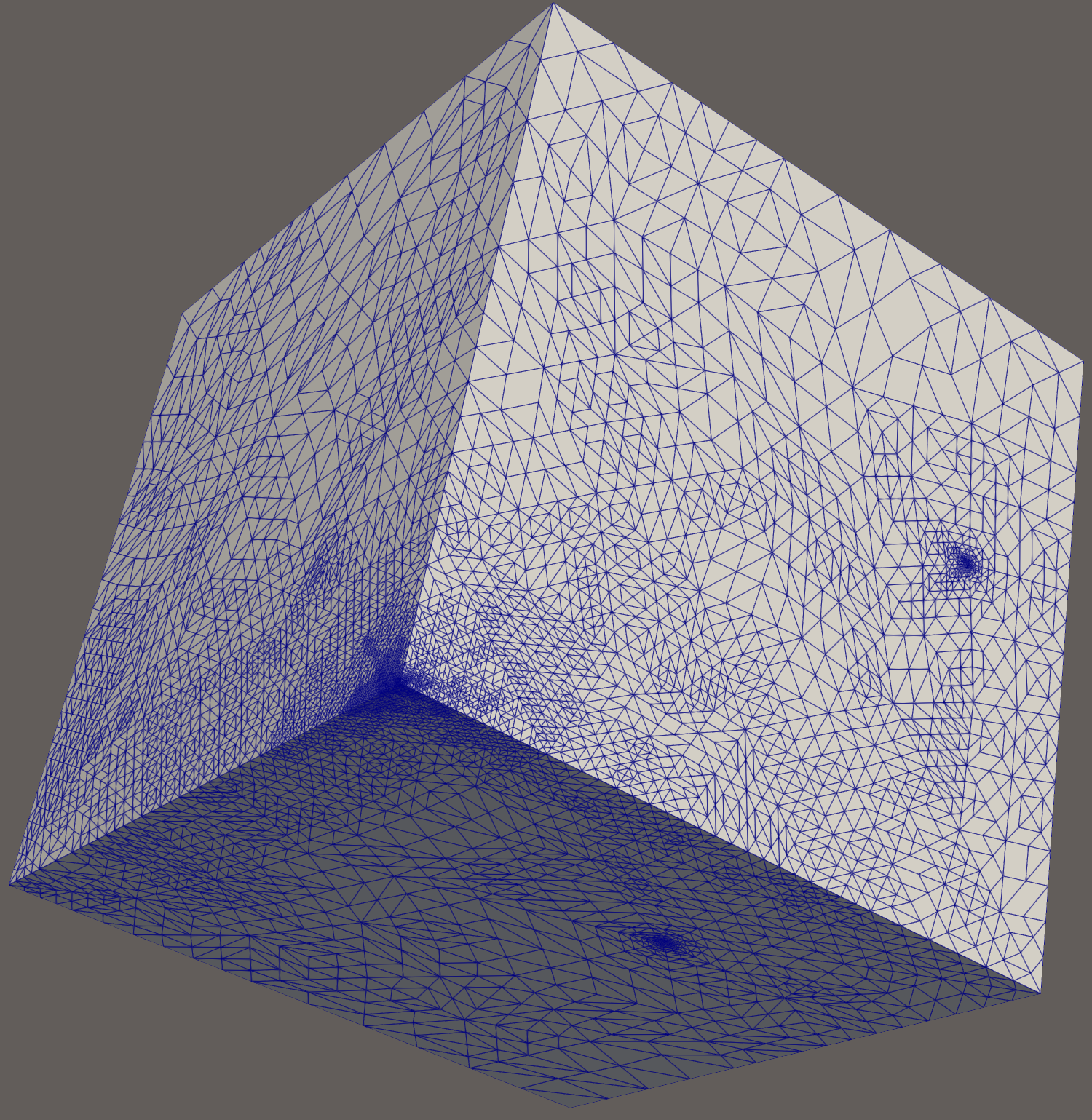}
    \includegraphics[width=0.28\linewidth]{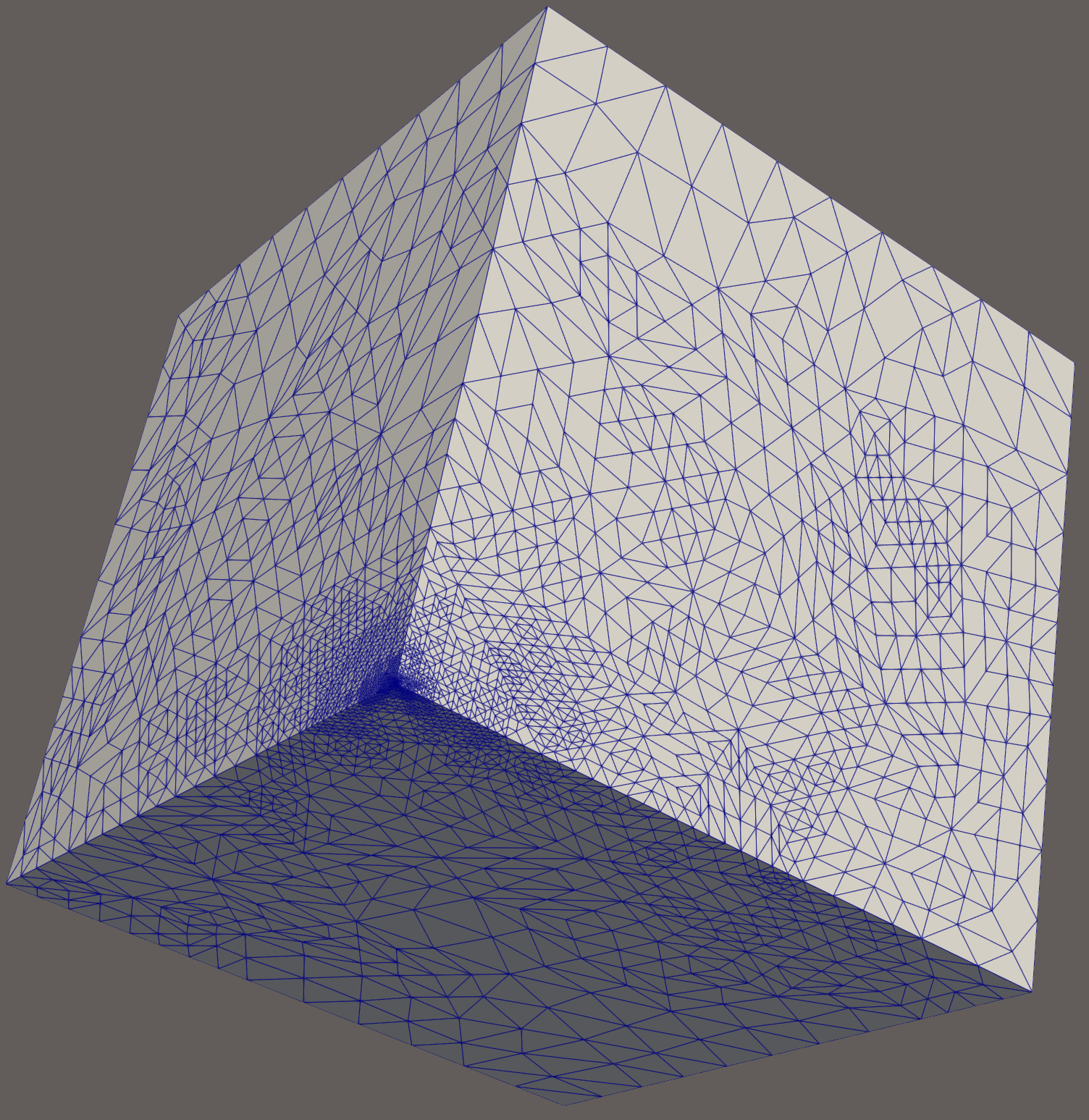}
    \caption{Convergence study for adaptive scheme \eqref{eq_adaptive_scheme} with estimator \eqref{def: estimator} on 3D corner singularity \eqref{eq_nonsmooth_3D}, with $\theta_{\rm{refine}}=0.2$, $p\in\{1,2\}$: left: convergence test; middle: last mesh for $p=1$; right: last mesh for $p=2$.}
    \label{fig:rate_3D}
\end{figure}

We observe from Figure \ref{fig:rate_3D} that the expected convergence rates $\mathcal{O}({\rm{DoFs}}^{-\frac{p}{3}})$  are obtained for both $p=1$ and $p=2$ cases, and meshes are mainly refined near the corner singularity. In addition, a stronger mesh grading phenomenon is observed for $p=2$ than $p=1$, which suggests the usage of high-order adaptive method \eqref{eq_adaptive_scheme} with estimator \eqref{def: estimator} for model problems with local singularity.

\section{Technical proofs}\label{sec:proofs}
This section contains the proofs of Lemmas \ref{lemma: dual norm of residual}, \ref{lemma: nonconforming error total} and Theorem \ref{theorem: lower bound}.

\subsection{Proof of Lemma \ref{lemma: dual norm of residual}}\label{sec: proof dual norm of residual}

In this section, we prove Lemma \ref{lemma: dual norm of residual}. Specifically,
we prove in Section~\ref{sec: proof BS interpolation} that,
\begin{equation} \label{eq:first_bnd_res}
\mu^{-1}\|\mathcal{R}_{\mesh}\|_{\bm{H}^{-1}(\Omega)}^2 \leq C(\Omega) \sum_{K\in\mesh}   \left\{
p \eta_{K,{\rm sta}}^2 +\eta_{K,{\rm nor}}^2+\eta_{K,{\rm tan},\Delta}^2+ \eta_{K,{\rm res}}^2+ \mathcal{O}(\bm{f})^2_K\right\}.
\end{equation}
and we prove in Section~\ref{sec: proof Nedelec interpolation} that
\begin{equation} \label{eq:second_bnd_res}
\mu^{-1}\|\mathcal{R}_{\mesh}\|_{\bm{H}^{-1}(\Omega)}^2 \leq C(\Omega) \sum_{K\in\mesh} p^2\left\{
 \eta_{K,{\rm sta}}^2 + \eta_{K,{\rm nor}}^2+ \mathcal{O}(\bm{f})^2_K\right\}.
\end{equation}
The combination of~\eqref{eq:first_bnd_res}--\eqref{eq:second_bnd_res} readily gives~\eqref{eq:dual_res1}.

Recalling the definition~\eqref{eq:def_res_dual} of the dual residual norm $\|\mathcal{R}_{\mesh}\|_{\bm{H}^{-1}(\Omega)}$, we need to bound $(\mu\nabla_{h} \bm{e}, \nabla \bm{v})_{\Omega}$ for all $\bm{v}\in \bm{H}^1_0(\Omega)\cap \bm{H}(\tdiv^0; \Omega)$.
First, using the weak forms  \eqref{eq:weak-div-free} and \eqref{eq:def-ah-lifting}, we have
\begin{align}
(\mu\nabla_{h} \bm{e}, \nabla \bm{v})_{\Omega}=
(\bm{f},\bm{v})_{\Omega} -  (\mu\nabla_{h} \bm{u}_h, \nabla \bm{v})_{\Omega}
=
(\bm{f},\bm{v})_{\Omega} -  a_h(\bm{u}_h, \bm{v}) + (\mu\mathcal{L}(\bm{u}_h),  \nabla \bm{v})_{\Omega}
.\label{eq:I-term}
\end{align}
Next, using the  discrete scheme \eqref{eq:discrete div-free}, we have for all $\bm{v}_{h} \in \bm{V}_{h0}^p\cap \bm{H}(\tdiv^0;\Omega)$
\begin{align}
(\bm{f},\bm{v}_{h})_{\Omega} -  a_h(\bm{u}_h, \bm{v}_{h})=0.\label{eq:discrete-error}
\end{align}
Combining the above two bounds, we have
\begin{equation}\label{def: err equation}
(\mu\nabla_{h} \bm{e}, \nabla \bm{v})_{\Omega}
= (\bm{f},  \bm{v}-  \bm{v}_h)_{\Omega} -
a_h(\bm{u}_h,  \bm{v}-\bm{v}_h)
+(\mu\mathcal{L}(\bm{u}_h),  \nabla \bm{v})_{\Omega}.
\end{equation}
We need now to select a suitable discrete test function $\bm{v}_{h} \in \bm{V}_{h0}^p\cap \bm{H}(\tdiv^0;\Omega)$.

\subsubsection{Bound using the Babu\v{s}ka--Suri interpolation operator} \label{sec: proof BS interpolation}

Using the generalized Bogovskiĭ operator \eqref{eq: Bogovskii operator}, we know there exists $\bm{\zeta}\in \bm{H}^2_0(\Omega)$ such that
\begin{equation}
\bm{v} = \curl{ \bm{\zeta}}.\label{eq: bound 1}
\end{equation}
Let $\mBS^p (\bm{\zeta})$ be the modified Babu\v{s}ka-Suri-operator defined in Lemma~\ref{lemma: Global BS} for each component of $\bm{\zeta}$. As $\mBS^p (\bm{\zeta})\in  \bm{P}_{p}(\mesh)\cap \bm{H}^1_{0}(\Omega)$, choosing $\bm{v}_{h} = \curl{ \mBS^p (\bm{\zeta})} \in \bm{V}_{h0}^p\cap \bm{H}(\tdiv^0;\Omega)$ in \eqref{def: err equation} and denoting  $\bm{\eta} =\bm{\zeta} -  \mBS^p (\bm{\zeta})$, with $\bm{\eta} \in \bm{H}^1_0(\Omega)\cap \bm{H}^2(\mesh)$, we have
\begin{align*}
(\mu\nabla_{h} \bm{e}, \nabla \bm{v})_{\Omega}&= (\bm{f},\curl{\bm{\eta}} )_{\Omega} -  a_h(\bm{u}_h,\curl{\bm{\eta}})
+(\mu\mathcal{L}(\bm{u}_h),  \nabla \curl{ \bm{\zeta}})_{\Omega}   .
\end{align*}
An integration by parts  implies
\begin{equation}
	\begin{split}
(\mu\nabla_{h} \bm{e}, \nabla \bm{v})_{\Omega}
=& \sum_{K\in \mesh}  (\bm{f} +\mu\Delta \bm{u}_h, \curl {\bm{\eta}})_K - \sum_{F\in \Fint} \mu (\jump{\nabla_h \bm{u}_h \bm{n}} , \{\curl{ \bm{\eta}} \})_F \\
&- \sum_{F\in \Fall} \mu (\{ \nabla_h\bm{u}_h \bm{n}\} , \jump{\curl {\bm{\eta}}})_F  -(\mu\mathcal{L}(\curl {\bm{\eta}}), \nabla_h \bm{u}_h)_{\Omega}
\\
&+(\mu\mathcal{L}(\bm{u}_h),  \nabla_h (\curl{\mBS^p (\bm{\zeta})}))_{\Omega}
-\sum_{F\in \Fall} ( \sigma \mu \bm{n}{\times}\jump{\bm{u}_h}, \bm{n}{\times}\jump{\curl{\bm{\eta}}})_F.
	\end{split}
\label{eq:astokes 1}
\end{equation}
Since $\curl{\bm{\zeta}}\in \bm{H}^1_0(\Omega)\cap \bm{H}(\tdiv^0;\Omega)$ and $\curl  {\mBS^p (\bm{\zeta})} \in \bm{V}_{h0}^p\cap \bm{H}(\tdiv^0;\Omega)$, we have $\jump{\curl {\bm{\eta}}}= -\jump{\curl {\mBS^p (\bm{\zeta})}} $, and also the identity
$$\sum_{F\in \Fall} \mu (\{ \nabla_h \bm{u}_h  \bm{n}\} , \jump{\curl {\bm{\eta}}})_F
= -\sum_{F\in \Fall} \mu (\{ \bm{n}{\times} (\nabla_h\bm{u}_h\bm{n})\} , \bm{n}{\times}  \jump{\curl {\mBS^p (\bm{\zeta})}})_F.
$$

Next, using the definition of lifting operator \eqref{def: lifting operator} and the fact $\mathcal{L}(\nabla {\times} \bm{\eta}) = - \mathcal{L}(\nabla {\times}\mBS^p (\bm{\zeta}))$, we have
\begin{equation}
	\begin{split}
(\mu\nabla_{h} \bm{e}, &\nabla \bm{v})_{\Omega}
= \sum_{K\in \mesh}  (\bm{f} +\mu\Delta \bm{u}_h, \curl {\bm{\eta}})_K  - \sum_{F\in \Fint} \mu (\jump{ \nabla_h \bm{u}_h \bm{n} }, \{\curl {\bm{\eta}} \})_F \\
&+(\mu\mathcal{L}(\bm{u}_h),  \nabla_h (\curl{\mBS^p (\bm{\zeta})}))_{\Omega}
-\sum_{F\in \Fall} ( \sigma \mu \bm{n}{\times}\jump{\bm{u}_h}, \bm{n}{\times}\jump{\curl {\bm{\eta}}})_F \\
=&   \!\!\sum_{K\in \mesh} (\curl{ (\bm{f} +\mu  \Delta \bm{u}_h)},  \bm{\eta})_K  -\!\! \sum_{F\in \Fint} \mu (\bm{n}{\times}\jump{\Delta_h \bm{u}_h},\bm{\eta})_F
- \sum_{F\in \Fint} \mu (\jump{ \nabla_h \bm{u}_h\bm{n}} , \{\curl{\bm{\eta}} \})_F
\\
&
+(\mu\mathcal{L}(\bm{u}_h),  \nabla_h (\curl{\mBS^p (\bm{\zeta})}))_{\Omega}
\!\!-\sum_{F\in \Fall} ( \sigma \mu \bm{n}{\times}\jump{\bm{u}_h}, \bm{n}{\times}\jump{(\curl {\bm{\eta}})})_F .	\end{split}
\label{eq:astokes 2}
\end{equation}
Invoking the Cauchy-Schwartz inequality, we infer
\begin{equation*}
\begin{split}
|(\mu\nabla_{h} \bm{e}, \nabla \bm{v})_{\Omega}|
\leq &\sum_{K\in \mesh}\bigg\{
\left( \frac{h_K}{p}\right)^2\|\curl{(\bm{f}+ \mu \Delta \bm{u}_h)}\|_{K} \left( \frac{h_K}{p}\right)^{-2}\|\bm{\eta}\|_{K}
\\
&+
\sum_{F\in \dK \cap \Fint}
\mu \left( \frac{h_F}{p}\right)^{\frac32} \|\bm{n}{\times}\jump{ \Delta_h \bm{u}_h}\|_{F} \left( \frac{h_F}{p}\right)^{-\frac32} \|\bm{\eta}\|_{F} \\
&+
\sum_{F\in \dK \cap \Fint}
\mu\left( \frac{h_F}{p}\right)^{\frac12} \|\jump{ \nabla_h \bm{u}_h \bm{n}}\|_{F} \left( \frac{h_F}{p}\right)^{-\frac12} \|\curl{\bm{\eta}}\|_{F} \\
& + \mu \|\mathcal{L}(\bm{u}_h)\|_{\Omega} \|\nabla_h (\curl{\mBS^p (\bm{\zeta})}) \|_{\Omega}
+
\sum_{F\in \dK}
\mu \|\sigma^{\frac12} \bm{n}{\times}\jump{\bm{u}_h}\|_F
\|\sigma^{\frac12} \bm{n}{\times}\jump{ \curl{\bm{\eta}}  }\|_F
\bigg\}.
\end{split}
\end{equation*}

Finally, using the approximation property \eqref{eq: hp-approximation BS}, the stability bounds \eqref{eq:lifting} and \eqref{def: penalty}, together with the estimate $| \bm{\zeta} |_{\bm{H}^2(\Omega)} \leq C(\Omega)  \| \nabla  \bm{v}\|_{\Omega}$ from \eqref{eq:stability-H2}, we obtain \eqref{eq:first_bnd_res},
\begin{equation}
	\begin{split}
|(\mu\nabla_{h} \bm{e}, \nabla \bm{v})_{\Omega}|
\leq C(\Omega)	\left\{ \sum_{K\in\mesh} \left\{
p \eta_{K,{\rm sta}}^2 +\eta_{K,{\rm nor}}^2+\eta_{K,{\rm tan},\Delta}^2+ \eta_{K,{\rm res}}^2+ \mathcal{O}(\bm{f})^2_K\right\}\right\}^{\frac12} \|\mu^{\frac12} \nabla  \bm{v}\|_{\Omega}.
\end{split}
\label{bounds I}
\end{equation}

\subsubsection{Bound using the N\'ed\'elec projection-based interpolation operator} \label{sec: proof Nedelec interpolation}

As $\bm{\zeta}\in \bm{H}^2_0(\Omega)$, let $\bm{I}_{\rm N}^{p} (\bm{\zeta})$ denote the N\'ed\'elec projection-based interpolation operator defined in Lemma~\ref{lemma: Global Nedelec}. Since $\bm{I}_{\rm N}^{p} (\bm{\zeta}) \in  \bm{N}_{p}(\mathcal{T}_h)\cap \bm{H}_0(\tcurl;\Omega)$, we choose
$\bm{v}_{h} = \curl{ \bm{I}_{\rm N}^{p} (\bm{\zeta}) }  = \bm{I}_{\rm RT}^{p}  (\curl{ \bm{\zeta}})\in \bm{V}_{h0}^p\cap \bm{H}(\tdiv^0;\Omega)$
in \eqref{def: err equation}. Denoting
$\bm{\eta} =\bm{\zeta} -  \bm{I}_{\rm N}^{p} (\bm{\zeta})$,
with $\bm{\eta} \in \bm{H}_0(\tcurl;\Omega)\cap \bm{H}^2(\mesh)$, and integrating by parts elementwise, we obtain
\begin{align*}
(\mu\nabla_{h} \bm{e}, \nabla \bm{v})_{\Omega}
&= \sum_{K\in \mesh}  (\bm{f} +\mu\Delta \bm{u}_h, \curl {\bm{\eta}})_K  - \sum_{F\in \Fint} \mu (\jump{ \nabla_h \bm{u}_h\bm{n} }, \{\curl {\bm{\eta}} \})_F \\
&+(\mu\mathcal{L}(\bm{u}_h),  \nabla_h (\curl{\bm{I}_{\rm N}^{p}  (\bm{\zeta})}))_{\Omega}
-\sum_{F\in \Fall} ( \sigma \mu \bm{n}{\times}\jump{\bm{u}_h}, \bm{n}{\times}\jump{\curl {\bm{\eta}}})_F.
\end{align*}
Since $\bm{u}_h|_K \in \bm{RT}_{p}(K)\subset \bm{P}_{p+1}(K)$, it follows that $\Delta \bm{u}_h|_K \in \bm{P}_{p-1}(K)$. Hence, invoking the orthogonality condition \eqref{eq:Nedelec-property4}, we obtain
$
(\bm{f} +\mu\Delta \bm{u}_h, \curl {\bm{\eta}})_K =  (\bm{f}, \curl {\bm{\eta}})_K
$
for all $K\in \mesh$.
Since $\bm{f}\in\bm{H}(\tcurl;\Omega)$ and $\bm{\eta} \in \bm{H}_0(\tcurl;\Omega)$, integrating by parts and invoking \eqref{eq:Nedelec-property3} yields
\begin{equation}\label{eq: integtation by parts curl}
\sum_{K\in \mesh}  (\bm{f} +\mu\Delta \bm{u}_h, \curl {\bm{\eta}})_K
=
(\bm{f}, \curl {\bm{\eta}})_{\Omega}
=
(\curl {\bm{f}},  \bm{\eta})_{\Omega}
=
\sum_{K\in \mesh}  ( \curl{\bm{f}} -\bm{\Pi}_K^{p-2} (\curl{\bm{f}}) ,  \bm{\eta})_K.
\end{equation}
Combining the above identities, we infer
\begin{equation}
	\begin{split}
(\mu\nabla_{h} \bm{e}, \nabla \bm{v})_{\Omega}
=&\sum_{K\in \mesh}  ( \curl{\bm{f}} -\bm{\Pi}_K^{p-2} (\curl{\bm{f}}) ,  \bm{\eta})_K  - \sum_{F\in \Fint} \mu (\jump{ \nabla_h \bm{u}_h  \bm{n} }, \{\curl {\bm{\eta}} \})_F \\
&+(\mu\mathcal{L}(\bm{u}_h),  \nabla_h (\curl{\bm{I}_{\rm N}^{p}  (\bm{\zeta})}))_{\Omega}
-\sum_{F\in \Fall} ( \sigma \mu \bm{n}{\times}\jump{\bm{u}_h}, \bm{n}{\times}\jump{\curl {\bm{\eta}}})_F
=:\sum_{j=1}^4T_j.
\end{split}
\label{eq:astokes 2 Nedelc}
\end{equation}
We will bound each terms $T_j$, $j=1,\dots,4$. Invoking the Cauchy-Schwartz inequality and  the approximation property \eqref{eq:Nedelec-hp} gives
\begin{equation}\label{bound  T1}
	\begin{split}
|T_1| \leq \!\!\!\sum_{K\in \mesh}  \| \curl{\bm{f}} -\bm{\Pi}_K^{p-2} (\curl{\bm{f}})\|_K  \| \bm{\eta}\|_K
& \leq  C \left\{\sum_{K\in \mesh}
 \left( \frac{h_K^2}{p}\right)^2  \| \curl{\bm{f}} -\bm{\Pi}_K^{p-2} (\curl{\bm{f}})\|_K^2\right\}^{\frac12}  \!\!| \bm{\zeta}|_{\bm{H}^2(\Omega)}.
\end{split}
\end{equation}
Next, we bound $\|\nabla (\curl{\bm{I}_{\rm N}^{p}  (\bm{\zeta})})\|_{K}$. Recalling the  property \eqref{eq:Nedelec-commut},
$
\curl{ \bm{I}_{\rm N}^{p} (\bm{\zeta})} = \bm{I}_{\rm RT}^{p}(\curl{ \bm{\zeta}}),
$
and using the triangle inequality, the inverse inequality \eqref{eq:inverse}, the identity $\bm{\Pi}_{K}^{p+1} \circ\bm{I}_{\rm RT}^{p} = \bm{I}_{\rm RT}^{p}$, the stability of the $L^2$-orthogonal projection \eqref{eq:L2projection-stability}, and the approximation property \eqref{eq:RT-hp}, we obtain
\begin{equation}\label{H1 stability of RT projection}
\begin{split}
\|\nabla &(\curl{\bm{I}_{\rm N}^{p}  (\bm{\zeta})})\|_{K}
\leq \|\nabla (\bm{I}_{\rm RT}^{p} (\curl{ \bm{\zeta}} ) -  \bm{\Pi}_{K}^{p+1}  (\curl{ \bm{\zeta}} ) )\|_{K}
+ \|\nabla \bm{\Pi}_{K}^{p+1}(\curl{ \bm{\zeta}}) \|_{K} \\
& \leq C \left( \frac{p^2}{h_K}\right) \|\bm{\Pi}_{K}^{p+1} (\bm{I}_{\rm RT}^{p} (\curl{ \bm{\zeta}} ) -    (\curl{ \bm{\zeta}}  ))\|_{K}
+ Cp^{\frac12}\|\nabla(\curl{ \bm{\zeta}}) \|_{K}
 \leq Cp  | \bm{\zeta}|_{\bm{H}^2(K)}.
\end{split}
\end{equation}
Combining the above estimate with the stability result \eqref{eq:lifting} yields
\begin{equation}\label{bound  T3}
	\begin{split}
|T_3| \leq \mu \|\mathcal{L}(\bm{u}_h)\|_{\Omega}  \|\nabla_h (\curl{\bm{I}_{\rm N}^{p}  (\bm{\zeta})})\|_{\Omega}
& \leq  C\mu \left\{ \sum_{K\in \mesh} \left(\frac{p^4}{h_F}\right)
 \| \bm{n}{\times}\jump{\bm{u}_h}\|_{\dK}^2
\right\}^{\frac12}  | \bm{\zeta}|_{\bm{H}^2(\Omega)}.
\end{split}
\end{equation}
Similarly, we estimate $\|\curl{\bm{\eta}}\|_{F}$. Recalling the commuting property \eqref{eq:Nedelec-commut},
$
\curl{\bm{\eta}} = \curl{\bm{\zeta}} - \bm{I}_{\rm RT}^{p} (\curl{ \bm{\zeta}} ),
$
and using the triangle inequality, the discrete trace inequality \eqref{eq: discrete trace}, the identity $\bm{\Pi}_{K}^{p+1} \circ\bm{I}_{\rm RT}^{p} = \bm{I}_{\rm RT}^{p}$, the stability of the $L^2$-orthogonal projection \eqref{eq:L2projection-stability}, and the approximation properties \eqref{eq:L2projection-trace} and \eqref{eq:RT-hp}, we obtain
\begin{equation}\label{trace approximation of RT projection}
\begin{split}
&\|\curl{\bm{\eta}}\|_{F}
\leq \| \curl{ \bm{\zeta}}  -  \bm{\Pi}_{K}^{p+1}  (\curl{ \bm{\zeta}} ) \|_{F}
+ \| \bm{I}_{\rm RT}^{p} (\curl{ \bm{\zeta}} ) -  \bm{\Pi}_{K}^{p+1}  (\curl{ \bm{\zeta}} ) \|_{F} \\
& \leq
 C\left( \frac{h_K}{p}\right)^{\frac12}\|\nabla(\curl{ \bm{\zeta}}) \|_{K}
+C \left( \frac{p}{h_K^{\frac12}}\right) \|\bm{\Pi}_{K}^{p+1} (\bm{I}_{\rm RT}^{p} (\curl{ \bm{\zeta}} ) -    (\curl{ \bm{\zeta}}  ))\|_{K}
 \leq Ch_K^{\frac12}  | \bm{\zeta}|_{\bm{H}^2(K)}.
\end{split}
\end{equation}
The above estimate yields
\begin{equation}\label{bound  T2}
	\begin{split}
|T_2| &\leq \mu \sum_{F\in \Fint} \|\jump{ \nabla_h\bm{u}_h \bm{n} }\|_{F} \|  \{ \curl {\bm{\eta}} \} \|_F
\leq C \mu \sum_{K\in \mesh} \left\{ h_K^{\frac12} \|\jump{ \nabla_h \bm{u}_h \bm{n} }\|_{\dK \cap \Fint}| \bm{\zeta}|_{\bm{H}^2(K)} \right\}\\
&
\leq  C\mu \left\{ \sum_{K\in \mesh} p\left(\frac{h_F}{p}\right)
 \| \jump{ \nabla_h \bm{u}_h \bm{n} }\|_{\dK \cap \Fint}^2
\right\}^{\frac12}  | \bm{\zeta}|_{\bm{H}^2(\Omega)}.
\end{split}
\end{equation}
Recalling the definition of $\sigma$ in \eqref{def: penalty}, we infer,
\begin{equation}\label{bound  T4}
	\begin{split}
|T_4| &\leq \mu \sum_{F\in \Fall} \| \sigma \bm{n}{\times}\jump{\bm{u}_h}\|_{F} \|  \bm{n}{\times}\jump{ \curl{\bm{\eta}} }\|_F
\leq C \mu \sum_{K\in \mesh} \left\{ h_K^{\frac12} \left(\frac{p^2}{h_K}\right)  \|\bm{n}{\times}\jump{\bm{u}_h} \|_{\dK}| \bm{\zeta}|_{\bm{H}^2(K)} \right\}\\
&
\leq  C\mu \left\{ \sum_{K\in \mesh} p^2\left(\frac{p^2}{h_F}\right)
 \| \bm{n}{\times}\jump{\bm{u}_h} \|_{\dK}^2
\right\}^{\frac12}  | \bm{\zeta}|_{\bm{H}^2(\Omega)}.
\end{split}
\end{equation}
Finally, using the estimate $| \bm{\zeta}|_{\bm{H}^2(\Omega)} \leq C(\Omega)  \| \nabla  \bm{v}\|_{\Omega}$ from \eqref{eq:stability-H2}, we obtain \eqref{eq:second_bnd_res},
\begin{equation*}
	\begin{split}
|(\mu\nabla_{h} \bm{e}, \nabla \bm{v})_{\Omega}|
\leq C(\Omega) \left\{\sum_{K\in\mesh} p^2\left\{
 \eta_{K,{\rm sta}}^2 + \eta_{K,{\rm nor}}^2+ \mathcal{O}(\bm{f})^2_K\right\}\right\}^{\frac12} |  \| \mu^{\frac12}\nabla  \bm{v}\|_{\Omega}.
\end{split}
\end{equation*}

\subsection{Proof of Lemma \ref{lemma: nonconforming error total}}\label{sec: proof full bound}

In this section, we prove Lemma \ref{lemma: nonconforming error total}, namely
$$
\|\mu^{\frac12}\nabla_h (\mathcal{R}_s(\bm{u_h}) -\bm{u}_h) \|_{\Omega}^2
\leq {} C \sum_{K\in \mathcal{T}_h} \left\{
\eta_{K,{\rm tan}}^2 + \eta_{K,{\rm sta}}^2
\right\}.
$$
Recalling the definition~\eqref{def: Stokes reconstruction global} applied with $\bm{w}_h:=\bm{u}_h$ and $\bm{w}_{\ba}:=\bm{u}_{\ba}$ and invoking the partition-of-unity property~\eqref{eq: PU}, we observe that
$$
\mathcal{R}_s(\bm{u_h})  - \bm{u}_{h} = \sum_{\ba \in\vertice} \{\bm{u}_{\ba} - \psi_{\ba} \bm{u}_h\}.
$$

Setting $\bm{\delta}_a: = \bm{u}_{\ba} - \psi_{\ba}\bm{u}_h$ for all $\ba \in \vertice$, and invoking the shape-regularity of the mesh, it is sufficient to prove that
\begin{equation}\label{bound enc_local}
\|\nabla_h \bm{\delta}_{\ba}\|_{\oma} \le {}
C \bigg\{
 \! \sum_{F\in \Fa   } \!
 \Big(\frac{h_F}{p}\Big) \|\n{\times}\jump{ \nabla_h \bm{u}_h} \|_{F}^2
+ \Big(\frac{p^2}{h_F} \Big)  \|\n{\times} \jump{\bm{u}_h}\|_{F}^2
\bigg\}^{\frac12}.
\end{equation}
where the set of mesh faces belonging to $\overline{\oma}$, say  $\bFa$, is partitioned as $\bFa=\Fa\cup \Fma$, where
$\Fa$ is the collection of the mesh faces which share $\ba$ and $\Fma$ the collection of all the mesh faces lying on $\partial \oma$ and not containing $\ba$. We notice that the two sets $\Fa$ and $\Fma$ are disjoint and that the set $\Fa$ contains mesh boundary faces if $\ba$ lies on the boundary.

Let us now prove~\eqref{bound enc_local}.  We decompose the proof of~\eqref{bound enc_local} in several steps.

(1)
Invoking the local Helmholtz decomposition \eqref{eq: HD identity} for $\underline{\bm{\theta}}: =  \nabla_h {\bm{\delta}_{\ba} }$ gives
\begin{align*}
 \nabla_h \bm{\delta}_{\ba} =   \nabla \bm{\phi}+\curlrw{ \underline{\bm{\beta}}}-q\underline{I},
\end{align*}
where $\bm{\phi}\in \bm{H}^1_{0}(\oma)\cap \bm{H}(\tdiv^0;\oma)$.

Using \eqref{eq:reconstruction-Stokes 2} we have $\nabla{\cdot} (\bm{u}_{\ba}- (\psi_{\ba} \bm{u}_h))=0$ in $L^2_0(\oma)$, then it gives $0 =(q,\nabla{\cdot} (\bm{u}_{\ba}- (\psi_{\ba} \bm{u}_h)  ))_{\oma}=(q\underline{I}, 	\nabla_h (\bm{u}_{\ba}- (\psi_{\ba} \bm{u}_h))_{\oma}$, therefore, it holds
\begin{align}\label{eq:ucuh}
	\|	\nabla_h \bm{\delta}_{\ba}\|_{\oma}^2=
(\nabla_h (\bm{u}_{\ba}- (\psi_{\ba} \bm{u}_h)), \nabla \bm{\phi})_{\oma}+(	\nabla_h  (\bm{u}_{\ba}- (\psi_{\ba} \bm{u}_h)), \curlrw{ \underline{\bm{\beta}}})_{\oma}.
\end{align}
Invoking \eqref{eq:difference-lifting}, \eqref{eq:lifting}, $\bm{\phi}\in \bm{H}^1_{0}(\oma)\cap \bm{H}(\tdiv^0;\oma)$, $\|{\psi}_{\ba}\|_{L^{\infty}(F)}\leq 1$ and the stability bound \eqref{eq: HD stability} and \eqref{eq:lifting}, we infer
\begin{equation}
\begin{split}
( \nabla_h (\bm{u}_{\ba}- (\psi_{\ba} \bm{u}_h)), \nabla \bm{\phi})_{\oma}
=  (\mathcal{L}(\psi_{\ba} \bm{u}_h),  \nabla \bm{\phi})_{\oma}
&\leq C \|\left(\frac{ p^2}{h}\right)^{\frac12} \jump{\psi_{\ba} \bm{u}_h} {\times} \n \|_{\Fa}\|\nabla \bm{\phi}\|_{\oma} \\
&
\leq C \|\left(\frac{ p^2}{h}\right)^{\frac12} \jump{\bm{u}_h}{\times} \n \|_{\Fa}\| \nabla_h \bm{\delta}_{\ba} \|_{\oma}.
\end{split}
\end{equation}
Next, we will derive the bound for the second term in \eqref{eq:ucuh}. As $ \underline{\bm{\beta}}\in \underline{\bm{H}}^1_{}(\oma)$,  let $\mKM (\underline{\bm{\beta}})$ denote the modified-Karkulik-Melenk interpolation operator defined in Lemma \ref{lemma: Local hp-KM} for each component of $\underline{\bm{\beta}}$ on the vertex patch $\oma$, we have
\begin{align*}
	(	\nabla_h \bm{\delta}_{\ba}, \curlrw{\underline{\bm{\beta}}})_{\oma}= (	\nabla_h \bm{\delta}_{\ba}, \curlrw{(\underline{\bm{\beta}}-\mKM(\underline{\bm{\beta}}})))_{\oma}
+(	\nabla_h \bm{\delta}_{\ba}, \curlrw{\mKM}(\underline{\bm{\beta}}))_{\oma}:= T_1 + T_2.
\end{align*}
The terms $T_1$ and $T_2$ are estimated independently.

(2) Bound on $T_1$. Introducing the column vectors $\{\bm{\beta}_j\}_{j\in\{1{:}d\}}$ composing the rows of $\underline{\bm{\beta}}$,    the scalar functions $\{u_{\ba,j}\}_{j\in\{1{:}d\}}$ composing the rows of  $\bm{u}_{\ba}$, $\{u_{h,j}\}_{j\in\{1{:}d\}}$ composing the rows of  $\bm{u}_{h}$ and recalling the definition of $\delta_{\ba}$, we have
$$
T_1  = \sumj (	\nabla_h (u_{\ba,j}- \psi_{\ba} u_{h,j}), \curl{(\bm{\beta}_j-\mKM (\bm{\beta}_j})))_{\oma},
$$
where we used the summation convention on repeated indices. We integrate by parts the curl operator in each tetrahedron $K\in \mesha$.  Since $\bm{u}_{\ba} \in \bm{H}^1_{0}(\oma)\cap \bm{H}(\tdiv^0;\oma) $, $\bm{n}_{\ba}{\times} \nabla_h u_{\ba,j}=\bm{0}$, on all the mesh faces composing $\partial\oma$, we have
$$
(\nabla_h u_{\ba,j}, \curl{(\bm{\beta}_j-\mKM (\bm{\beta}_j})))_{\oma}=0.
$$

Moreover, since $\bm{\beta}_j- \mKM (\bm{\beta}_j) \in \bm{H}^1(\oma)$ is single-valued on all faces $F \in  \Fa$ and $\psi_{\ba} u_{h,j}|_F = 0$ for all faces $F\in \Fma$, we obtain
\begin{align}\label{eqn1_nonconf}
T_1
&=  \sum\limits_{F \in \Fa} \sumj
( \bm{n} {\times} \jump{\nabla_h (\psi_{\ba}u_{h,j})}, \bm{\beta}_j - \mKM (\bm{\beta}_j) )_F.
\end{align}
Using the Cauchy--Schwarz inequality, we obtain
$$
|T_{1}| \leq \bigg\{\sum\limits_{F \in {\Fa}} \sumj \frac{h_F}{p} \|\bm{n} {\times} \jump{\nabla_h(\psi_{\ba}u_{h,j})}\|_{F}^2\bigg\}^{\frac{1}{2}}
\bigg\{\sum\limits_{F \in {\Fa}} \sumj \frac{p}{h_F} \| \bm{\beta}_j - \mKM (\bm{\beta}_j)\|_{F}^2\bigg\}^{\frac{1}{2}}
$$
Invoking the shape-regularity of the mesh, the approximation estimate in Lemma \ref{lemma: Local hp-KM}, and the stability of the local Helmholtz decomposition (see~\eqref{eq: HD stability}), we infer that
\begin{align}\label{eq:eta-face-bound}
\bigg\{\sum\limits_{F \in {\Fa}} \sumj \frac{p}{h_F} \| \bm{\beta}_j - \mKM (\bm{\beta}_j)\|_{F}^2\bigg\}^{\frac12}
\leq C  | \underline{\bm{\beta}}|_{\underline{\bm{H}}^1(\oma)}^2
\leq C \| \nabla_h {\bm{\delta}_{\ba}} \|_{\oma}.
\end{align}

Since $\bm{n} {\times} \nabla \psi_{\ba}$ is single-valued at $F$,  $\|{\psi}_{\ba}\|_{L^{\infty}(F)}\leq 1$, $\|\nabla {\psi}_{\ba}\|_{L^{\infty}(F)}\leq Ch_F^{-1}$, and also $\bm{u}_h \in \bm{H}(\tdiv^0;\Omega)$,  we infer that
\begin{align*}
&\sum\limits_{F \in {\Fa}} \sumj \frac{h_F}{p}
 \|\bm{n} {\times} \jump{\nabla_h (\psi_{\ba}u_{h,j})}\|_{F}^2\\
& \leq \sum\limits_{F \in {\Fa}} \sumj \frac{h_F}{p}
 \left\{
 \|\bm{n} {\times} \jump{\nabla_h u_{h,j}}\|_{F}^2 \|\psi_{\ba}\|_{L^{\infty}(F)}^2
 +
 \|\jump{u_{h,j}}\|_{F}^2 \|\bm{n} {\times}  \nabla \psi_{\ba}\|_{L^{\infty}(F)}^2
  \right\}\\
& \leq C
\sum_{F\in \Fa}
 \left\{
\frac{h_F}{p} \|\n {\times} \jump{\nabla_h \bm{u}_h}\|_{F}^2
+
\frac{1}{h_F p} \|\n {\times} \jump{\bm{u}_h}\|_{F}^2
\right\}.
\end{align*}
Combining the above bounds and using $p\geq1$, we obtain
\begin{align}\label{eq: bound T1}
|T_1|
& \leq C \left\{
\sum_{F\in \Fa}
\frac{h_F}{p} \|\n {\times} \jump{\nabla_h \bm{u}_h}\|_{F}^2
+
\frac{p^2}{h_F} \|\n {\times} \jump{\bm{u}_h}\|_{F}^2
\right\}^{\frac12}
 \|  \nabla_h {\bm{\delta}_{\ba}}  \|_{\oma}.
\end{align}

(3) Bound on $T_2$.  Since  $\mKM (\underline{\bm{\beta}}) \in \underline{\bm{H}}^1(\oma)$, $\jump{\curlrw{\mKM (\underline{\bm{\beta}} }) \n}_F = \bm{0}$ across every interface $F\in \Fa$, and $\divrw (\curlrw{\mKM}(\underline{\bm{\beta})})=  \bm{0}$ , integrating by parts the gradient operator gives
$$
(	\nabla \bm{u}_{\ba}, \curlrw{\mKM}(\underline{\bm{\beta}} ))_{\oma}=0.
$$
Recalling that $\psi_{\ba}$ vanishes on all faces $F\in\Fma$ and observing that $\curlrw{\mKM(\underline{\bm{\psi}})}\bm{n}$ is  single-valued on those faces.
\begin{align*}
(	\nabla_h (\psi_{\ba} \bm{u}_h), \curlrw{ \mKM(\underline{\bm{\beta}} )})_{\oma}
& = \sum_{F\in \Fa}( \jump{\psi_{\ba} \bm{u}_h},  \curlrw {\mKM(\underline{\bm{\beta}} )}\bm{n})_F.
\end{align*}
Combining the above two bounds, we have
\begin{align*}
T_2
& = -\sum_{F\in \Fa}( \jump{\psi_{\ba} \bm{u}_h},  \curlrw{ \mKM(\underline{\bm{\beta}} )}\bm{n})_F.
\end{align*}
For every face $F\in \Fa$, we can pick up a mesh cell $K$ of which $F$ is a face and obtain from \eqref{eq: discrete trace} that
$$
\|\curl{ \mKM (\underline{\bm{\beta}})}\bm{n}\|_{F}^2\leq C \frac{p^2}{h_F}\|\curlrw { \mKM (\underline{\bm{\beta}})}\|_{K}^2.
$$
Next, using the Cauchy-Schwarz inequality and the above bounds, we infer
\begin{align*}
|T_2|
& \leq  \sum_{F\in \Fa} \|\jump{\psi_{\ba} \bm{u}_h}\|_{F} \|\curlrw { \mKM (\underline{\bm{\beta}})}\bm{n}\|_{F}
 \leq C \left\{
\sum_{F\in \Fa} \frac{p^2}{h_F} \| \jump{\bm{u}_h}\|_{F}^2
\right\}^{\frac12}
  \|\curlrw { \mKM(\underline{\bm{\beta}})}\|_{\oma}.
\end{align*}
Owing to \eqref{eq: hp-approximation} and  \eqref{eq: HD stability}, we infer that
\begin{equation}\label{relation 2}
  \|\curlrw{ \mKM (\underline{\bm{\beta}})}\|_{\oma}\leq   \|\nabla_h \mKM (\underline{\bm{\beta}})\|_{\oma}
\leq C\| \nabla  \underline{\bm{\beta}}\|_{\oma}\leq C  \|\nabla_h {\bm{\delta}_{\ba}} \|_{\oma}.
\end{equation}
Combining the above bounds, and noticing that $ \| \jump{\bm{u}_h}\|_{F}=  \|\n {\times} \jump{\bm{u}_h}\|_{F} $, we obtain
\begin{align}\label{eq: bound T2}
|T_2|
& \leq C \left\{
\sum_{F\in \Fa} \frac{p^2}{h_F} \| \n {\times} \jump{\bm{u}_h}\|_{F}^2
\right\}^{\frac12}
\|  \nabla_h {\bm{\delta}_{\ba}} \|_{\oma}.
\end{align}
Finally, combining bounds of \eqref{eq: bound T1}, \eqref{eq: bound T2}, together with \eqref{eq:ucuh} and invoking the triangle inequality,  we derived the bound \eqref{bound enc_local}.

\subsection{Proof of Theorem \ref{theorem: lower bound}}\label{sec: proof lower bound}

Let $b_K$ denote the standard bubble function associated with the mesh cell $K$.  The function $b_K$ vanishes on $\partial K$ and is defined as the product of the affine basis functions corresponding to the vertices of  $K$.
To prove the lower bounds, we first recall some useful results.
\begin{lemma}[Inverse estimate with bubbles] Let $K\in\mesh$. For all $q_p\in \mathbb{P}_p(K)$, the following holds:
	\begin{align}
		\|b_K^{\frac{\alpha}{2}} q_p\|_{K}\leq C(p) \|b_K^{\frac{\beta}{2}} q_p\|_{K}\quad -\frac{1}{2}<\alpha\leq \beta.\label{eq:inverse-bubble}
	\end{align}

\end{lemma}

\begin{proof}
	The proof can be found in \cite[Proposition~3.6]{Dong21}.
\end{proof}

The following result is proven in \cite[Proposition~3.85, 3.86]{Verfurth2013posteriori} for $d=2$ and $d=3$.
\begin{lemma}[Polynomial weighted $H^1$ to $L^2$ inverse estimate] Let $K\in\mesh$. Then, there exists a positive constant $C(p)>0$, independent of $h$, such that for all $q_p\in \mathbb{P}_p(K)$,
	\begin{align}
		\|\nabla (b_K q_p)\|_{K}\leq C(p) h_K^{-1} \|q_p\|_{K}.\label{eq:bubble-inverse}
	\end{align}

\end{lemma}

The following result can be found in \cite[Corollary~3.9]{Dong21}.

\begin{lemma}[Polynomial extension stability result]\label{lemma:extension} Let $K\in\mesh$ and $F\subset \partial K$. There exists an extension operator $E:\mathbb{P}_p(F)\rightarrow H^2(K)$ such that there exists a constant $C(p)>0$, independent of $h$ but depending on $p$, such that for all $q_p\in \mathbb{P}_p(F)$,
	\begin{align}
		&E(q_p)_{|F}=b_{\widetilde{F}} {q_p}_{|F},\\
		&	\|E(q_p)\|_{K}+ h_K|E(q_p)|_{H^1(K)}+h_K^2 | E(q_p)|_{H^2(K)}\leq C(p) h_K^{1/2} \|q_p\|_{F},\label{eq:extension-stability}
	\end{align}
	where $b_{\widetilde{F}}$ is the standard face bubble function associated with $F$.

\end{lemma}

\begin{proof}[Proof of \eqref{eq:error-cell-residual}--\eqref{eq:error-tangential}](1) Error representation formula. We
apply \eqref{eq:weak-div-free} and integrate by parts the vector Laplace operator for any $\bm{v}\in \bm{H}^1_0(\Omega)\cap \bm{H}(\tdiv^0;\Omega)$ to get the following error equation
	\begin{equation}
		\begin{split}
			(\mu \nabla_h\bm{e},\nabla \bm{v})_{\Omega}&=(\bm{f},\bm{v})_{\Omega}-(\mu \nabla_h\bm{u}_h,\nabla\bm{v})_{\Omega}=(\bm{f}+\mu \Delta_h\bm{u}_h,\bm{v})_{\Omega}-\sum_{K\in \mesh} (\mu \nabla\bm{u}_h\bm{n},\bm{v})_{\partial K}\\
			&=\sum_{K\in \mesh} (\bm{f}+\mu \Delta\bm{u}_h,\bm{v})_K-\sum_{F\in \Fint} (\mu \jump{\nabla_h\bm{u}_h \bm{n}}, \bm{v})_F.
		\end{split}
		\label{eq:error-second-order}
	\end{equation}
	For $\bm{w}\in \bm{H}^2_0(\Omega)$, we have $\nabla{\times} \bm{w}\in \bm{H}^1_0(\Omega)\cap \bm{H}(\tdiv^0;\Omega)$, choosing $\bm{v}:=\nabla{\times} \bm{w}$ in \eqref{eq:error-second-order}, and integrating by parts  the curl operator gives the following error representation formula:
	\begin{equation}
		\begin{split}
			(\mu \nabla_h\bm{e},\nabla (\nabla{\times} \bm{w}))_{\Omega}&=\sum_{K\in \mathcal{T}_h} (\bm{f}+\mu \Delta\bm{u}_h,\nabla{\times} \bm{w})_K-\sum_{F\in \Fint} (\mu \jump{(\nabla_h\bm{u}_h \bm{n})},\nabla{\times} \bm{w})_F
			\\&=\sum_{K\in \mesh}(\nabla{\times} (\bm{f}+\mu \Delta\bm{u}_h),\bm{w})_K
            -\sum_{F\in \Fint} (\mu \jump{(\nabla_h\bm{u}_h \bm{n})},\nabla{\times} \bm{w})_F\\
			&-\sum_{F\in \Fint}(\mu \bm{n}{\times} \jump{\Delta_h\bm{u}_h},\bm{w})_F=:\sum_{i=1}^3 T_i.
		\end{split}
		\label{eq:error-fourth}
	\end{equation}

	(2) Proof of \eqref{eq:error-cell-residual}.
	Choosing  $\bm{w}:=b_K^2(\bm{\Pi}_{K}^{p-2}\nabla {\times} \bm{f}+\mu \nabla{\times}\Delta \bm{u}_h)\in \bm{H}^2_0(K)$ in \eqref{eq:error-fourth}, then the second and third terms on the right-hand side of \eqref{eq:error-fourth} vanish, we have from the Cauchy-Schwarz inequality and \eqref{eq:inverse} that
	\begin{equation}
		\begin{split}
			\|b_K (\bm{\Pi}_{K}^{p-2}\nabla {\times} \bm{f}+&\mu \nabla{\times}\Delta \bm{u}_h)\|_{K}^2=-(\mu \nabla \bm{e},\nabla (\nabla {\times} \bm{w}))_K-(\nabla{\times} \bm{f}-\bm{\Pi}_{K}^{p-2}\nabla{\times}\bm{f}, \bm{w})_K\\
			&\leq {C(p)}\Big( \frac{\mu}{h_K}\|\nabla \bm{e}\|_{K} \|\nabla {\times} \bm{w}\|_{K}+\|\nabla{\times} \bm{f}-\bm{\Pi}_{K}^{p-2}\nabla {\times}\bm{f}\|_{K}\|\bm{w}\|_{K}\Big).
		\end{split}\label{eq:bKcurlf}
	\end{equation}
It follows from	\eqref{eq:bubble-inverse} that
	\begin{align}
		\|\nabla{\times} \bm{w}\|_{K}&=\|\nabla{\times} (b_K^2(\bm{\Pi}_{K}^{p-2}\nabla {\times} \bm{f}+\mu \nabla{\times}\Delta \bm{u}_h))\|_{K}\leq  \frac{C(p)}{h_K} \|b_K (\bm{\Pi}_{K}^{p-2}\nabla {\times} \bm{f}+\mu \nabla {\times}\Delta \bm{u}_h)\|_{K}.\label{eq:curlw}
	\end{align}
	Combining \eqref{eq:bKcurlf}, \eqref{eq:curlw} and $\|\bm{w}\|_{K}\leq C\|b_K (\bm{\Pi}_{K}^{p-2}\nabla {\times} \bm{f}+\mu\nabla {\times}  \Delta \bm{u}_h)\|_{K}$ owing to $b_K\leq 1$ gives
	\begin{align}
		\|b_K (\bm{\Pi}_{K}^{p-2}\nabla {\times} \bm{f}+\mu \nabla {\times} \Delta \bm{u}_h)\|_{K}\leq C(p)\Big(\frac{\mu}{h_K^2}  \|\nabla \bm{e}\|_{K}+\|\nabla{\times} \bm{f}-\bm{\Pi}_{K}^{p-2}\nabla{\times}\bm{f}\|_{K}\Big).\label{eq:cell-bK}
	\end{align}
	An application of \eqref{eq:inverse-bubble} with $\alpha=0$ and $\beta=2$ gives
	\begin{equation}
		\begin{split}
			\|\bm{\Pi}_{K}^{p-2}\nabla {\times} \bm{f}+\mu \nabla {\times} \Delta \bm{u}_h\|_{K}&\leq C(p)\Big(\frac{\mu}{h_K^2}   \|\nabla \bm{e}\|_{K}+\|\nabla{\times} \bm{f}-\bm{\Pi}_{K}^{p-2}\nabla {\times} \bm{f}\|_{K}\Big),
		\end{split}\label{eq:residual-curlf}
	\end{equation}
	which yields \eqref{eq:error-cell-residual}.
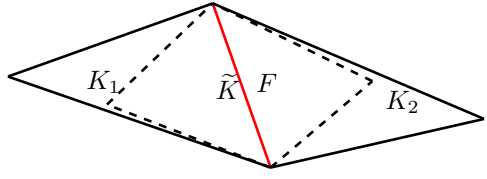
\begin{figure}
	\centering
	\begin{tikzpicture}[xscale=0.4, yscale=0.2, thick,rotate=100]

		\draw[line width=1pt] (-6.53, 0.05) -- (-4.59, -7.45);
		\draw[line width=1pt] (-6.53, 0.05)  -- (0.9, 7.55);
		\draw[line width=1pt,red] (-6.53, 0.05)  -- (4.49, 0.05);
		\draw[line width=1pt] (-4.59, -7.45) -- (4.49, 0.05);
		\draw[line width=1pt] (0.9, 7.55)  -- (4.49, 0.05);

		\draw[dashed, line width=1pt] (-6.53, 0.05)  -- (-1.49, -4.25);
		\draw[dashed, line width=1pt] (-1.49, -4.25)  -- (4.49, 0.05);
		\draw[dashed, line width=1pt] (-1.49, 4.65) -- (4.49, 0.05);
		\draw[dashed, line width=1pt] (-6.53, 0.05) -- (-1.49, 4.65);

		\node at (0, 4.5) {$K_1$};               
		\node at (-3, -5) {$K_2$};              
		\node at (-1, 0.5) {$\widetilde{K}$};        
		\node at (-1, -0.8) {$F$};               
	\end{tikzpicture}
	\caption{An illustration of a kite for triangular meshes. The solid lines form the union of two triangular cells $K_1$ and $K_2$; the dashed lines form the kite $\widetilde{K}$, which is symmetric with respect to the common face $F$.}
	\label{fig:kite}
\end{figure}

	(3) Proof of  \eqref{eq:error-normal-flux}. We decompose the estimator \eqref{eq:error-normal-flux} into its tangential and normal components,
\begin{align}
\|\jump{\nabla_h \bm{u}_h\bm{n}_F}\|_{F}^2
=
\| \bm{n}_F{\times}\jump{\nabla_h \bm{u}_h\bm{n}_F} \|_{F}^2
+
\| \bm{n}_F{\cdot}\jump{\nabla_h \bm{u}_h\bm{n}_F} \|_{F}^2.\label{eq:normal-tangential-decompose}
\end{align}
We estimate these two contributions separately.

	(i) The estimate of the first term on the right-hand side of \eqref{eq:normal-tangential-decompose} relies on a bespoke bubble function constructed in \cite{Dong21}. For a given face $F\in \Fint$, there exists a kite domain $\widetilde{K}\subset K_1\cup K_2$, where $K_1, K_2\in\mesh$ are the two cells sharing the common face $F$, such that $\widetilde{K}$ is symmetric with respect to $F$; see Figure~\ref{fig:kite} for an illustration in two dimensions. An analogous three-dimensional construction can be obtained for tetrahedral meshes, where the kite takes the form of a hexahedral patch.

	We define $b_{\widetilde{F}}$ as the standard face bubble function on $\widetilde{K}$, given by the product of the nodal linear basis functions associated with the vertices of $F$ on each tetrahedron forming the kite $\widetilde{K}$. Owing to the symmetry of $\widetilde{K}$ with respect to $F$, the normal derivative is continuous across $F$, namely,
	$
	\jump{\nabla b_{\widetilde{F}}{\cdot}\bm{n}}_{F}=0.
	$
	Let $b_\ell$ be an affine function on $\widetilde{K}$ satisfying
	$
	{b_\ell}_{|F}=0,
	\bm{n}_F{\cdot}\nabla b_\ell=\frac{1}{h_F}.
	$
	We then define
	$
	b_F:=b_\ell b_{\widetilde{F}}^2 \in C^1(\Omega)\cap H^2_0(\Omega)
	$.
	The following result holds.
	\begin{align*}
{b_F}|_{\widehat{F}}=0,
\quad  {\jump{\nabla b_F}}_{\widehat{F}}=\bm{0} \quad \forall \widehat{F}\in \Fall,\; (\bm{n}_F{\cdot} \nabla b_F)_{\widehat{F}}=0\;\forall \widehat{F}\in \Fall\backslash F,\quad \bm{n}_F{\cdot} \nabla b_F=\frac{1}{h_F} b_{\widetilde{F}}^2\quad \mbox{on}\;F.
	\end{align*}
	We define $\bm{w}:=b_F \bm{r}\in \bm{H}^2_0(\widetilde{K})$, where $\bm{r}:= E(\bm{n}_F{\times} \jump{\nabla_h\bm{u}_h \bm{n}_F})$ and $E$ is the extension operator defined in Lemma~\ref{lemma:extension} associated with $\widetilde{K}$.
Taking $\bm{w}:=b_F \bm{r}\in \bm{H}^2_0(\widetilde{K})$	in the error equation \eqref{eq:error-fourth} yields
	\begin{align}
		(\mu \jump{(\nabla_h\bm{u}_h \bm{n}_F)}, \nabla{\times} \bm{w})_F=-(\nabla{\times} (\bm{f}+\mu \Delta_h \bm{u}_h), \bm{w})_{\widetilde{K}}-(\mu \nabla_h\bm{e}, \nabla(\nabla{\times} \bm{w}))_{\widetilde{K}}.\label{eq:error-I1I2}
	\end{align}
Decomposing $\nabla{\times} \bm{w}$ into its normal and tangential components, and using the fact that $\bm{w}|_F=(b_F \bm{r})|_F=0$ on all faces $F\in \Fint$ since $b_F|_F=0$, we infer that $\bm{n}_F{\cdot}(\nabla{\times} \bm{w})=0$. Hence,
\begin{equation}
	\begin{split}
	 \mu(\jump{(\nabla_h\bm{u}_h \bm{n}_F)}, \nabla{\times} \bm{w})_F
	 =\mu(\bm{n}_F{\times}  \jump{\nabla_h\bm{u}_h \bm{n}_F}, \bm{n}_F{\times}(\nabla{\times} \bm{w}))_F.
	\label{eq:w}
	\end{split}
\end{equation}
Using the vector identity $
	\bm{n} {\times} (\nabla b_F {\times} \bm{r}) = \nabla b_F \, (\bm{n} {\cdot} \bm{r}) - \bm{r} \, (\bm{n} {\cdot} \nabla b_F)
	$, the following holds
	\begin{align*}
		\bm{n}_F{\times} (\nabla {\times} \bm{w})|_{F}=\bm{n}_F{\times} \Big( b_F \nabla {\times} \bm{r}+\nabla b_F {\times} \bm{r}\Big)|_{F}
=\nabla b_F(\bm{n}_F{\cdot}\bm{r})-\bm{r} (\bm{n}_F{\cdot} \nabla b_F)|_{F}+b_F (\bm{n}_F{\times} (\nabla {\times} \bm{r}))|_{F}.
	\end{align*}
	Using the fact that ${(b_F)}|_{F}=0$ and thus $(\bm{n}_F{\times} \nabla b_F)|_{F}=\bm{0}$, and decomposing $\nabla b_F$ into its normal and tangential components, we can obtain
$
		\bm{n}_F{\times} (\nabla {\times} \bm{w})|_{F}=\bm{n}_F (\nabla b_F{\cdot}\bm{n}_F)(\bm{n}_F{\cdot}\bm{r})|_{F}-(\bm{r} \bm{n}_F{\cdot} \nabla b_F)|_{F}.
$
This combined with \eqref{eq:w} yields
	\begin{align*}
			(\mu \jump{ \nabla_h\bm{u}_h \bm{n}_F }, \nabla{\times} \bm{w})_F&=\mu (\bm{n}_F{\times}  \jump{\nabla_h \bm{u}_h\bm{n}_F}, \bm{n}_F (\nabla b_F{\cdot}\bm{n}_F)(\bm{n}_F{\cdot}\bm{r}))_F\\
            &\quad -\mu(\bm{n}_F{\times}  \jump{\nabla_h\bm{u}_h \bm{n}_F}, \bm{r} (\bm{n}_F{\cdot} \nabla b_F))_F.
	\end{align*}
Invoking $(\bm{n}_F{\times} \jump{\nabla_h \bm{u}_h\bm{n}_F}){\cdot} \bm{n}_F = \jump{\nabla_h\bm{u}_h \bm{n}_F} (\bm{n}_F{\times} \bm{n}_F) = 0$, the first term on the right-hand vanishes. Thus, we have
$
			(\mu \jump{\nabla_h \bm{u}_h\bm{n}_F}, \nabla{\times} \bm{w})_F=-\mu (\bm{n}_F{\times}  \jump{\nabla_h \bm{u}_h\bm{n}_F}, \bm{r} (\bm{n}_F{\cdot} \nabla b_F))_F.
$
	In view of $(\bm{n}{\cdot} \nabla b_F)_{F}=\frac{1}{h_F}b_{\widetilde{F}}^2$, we deduce
	\begin{align}
			(\mu \jump{\nabla_h \bm{u}_h\bm{n}_F}, \nabla{\times} \bm{w})_F=-\mu(\bm{n}_F{\times} \jump{\nabla_h\bm{u}_h \bm{n}_F}, \bm{r}b_{\widetilde{F}}^2 )_F h_F^{-1}.\label{eq:I-modified-form}
	\end{align}
	Plugging \eqref{eq:I-modified-form}	into the error equation \eqref{eq:error-I1I2}, we obtain
	\begin{equation}
		\begin{split}
			&\mu (\bm{n}_F{\times} \jump{\nabla_h\bm{u}_h \bm{n}_F}, \bm{r}b_{\widetilde{F}}^2 )_F =h_F \Big((\mu  \nabla_h \bm{e}, \nabla(\nabla{\times} \bm{w}))_{\widetilde{K}}+(\nabla{\times} (\bm{f}+\mu \Delta \bm{u}_h), \bm{w})_{\widetilde{K}}\Big)\\
			&=h_F \Big((\mu  \nabla_h\bm{e}, \nabla(\nabla{\times} (b_\ell b_{\widetilde{F}}^2 \bm{r})))_{\widetilde{K}}+(\nabla{\times} (\bm{f}+\mu \Delta_h \bm{u}_h), b_\ell b_{\widetilde{F}}^2 \bm{r})_{\widetilde{K}}\Big)
=:T_1+T_2.
		\end{split}
		\label{eq:edge-normal-expression}
	\end{equation}
We first bound $T_1$. Using the bounds \eqref{eq:inverse},  \eqref{eq:extension-stability}, the boundedness of $b_F$, $\|\nabla b_F\|_{L^\infty(\widetilde{K})}\leq C h_F^{-1}$ and   $| b_F|_{W^{2,\infty}(\widetilde{K})}\leq C h_F^{-2}$, we deduce that
	\begin{equation}
		\begin{split}
	\|\nabla(\nabla{\times} &(b_\ell b_{\widetilde{F}}^2 \bm{r}))\|_{\widetilde{K}}=	\|\nabla(\nabla{\times} (b_F\bm{r}))\|_{\widetilde{K}}\leq | b_F\bm{r}|_{\bm{H}^2(\widetilde{K})}\\
			&\leq C\Big(\|b_F\|_{L^\infty(\widetilde{K})} | \bm{r}|_{\bm{H}^2(\widetilde{K})}+\|\nabla b_F\|_{\bm{L}^\infty(\widetilde{K})} |\bm{r}|_{\bm{H}^1(\widetilde{K})}+ | b_F|_{W^{2,\infty}(\widetilde{K})}\|\bm{r}\|_{\widetilde{K}}\Big)\\
			&\leq C \Big(| \bm{r}|_{\bm{H}^2(\widetilde{K})}+h_F^{-1} | \bm{r}|_{\bm{H}^1(\widetilde{K})}+h_F^{-2}\|\bm{r}\|_{\widetilde{K}}\Big)
\leq C(p) h_F^{-3/2} \|\bm{n}_F{\times} \jump{\nabla_h\bm{u}_h \bm{n}_F}\|_{F}.
		\end{split}
		\label{eq:curlcurl-r}
	\end{equation}
	Therefore, applying the Cauchy-Schwarz inequality to $T_1$, we obtain
	\begin{align*}
		T_1\leq C(p)h_F^{-1/2} \|\mu \nabla_h\bm{e}\|_{\widetilde{K}}  \|\bm{n}_F{\times} \jump{\nabla_h\bm{u}_h \bm{n}_F}\|_{F}.
	\end{align*}
Now we estimate $T_2$.	Invoking \eqref{eq:residual-curlf} and \eqref{eq:extension-stability} leads to
	\begin{equation}
		\begin{split}
			T_2&=h_F(b_{\widetilde{F}}^2 \nabla_h{\times} (\bm{f}+\mu \Delta_h \bm{u}_h), b_\ell \bm{r})_{\widetilde{K}}\leq h_F\|\nabla_h{\times} (\bm{f}+\mu \Delta_h \bm{u}_h)\|_{\widetilde{K}} \|\bm{r}\|_{\widetilde{K}}\\
			&\leq C \Big(h_F^{-\frac12}\mu\|\nabla_h\bm{e}\|_{\widetilde{K}}+h_F^{3/2}\|\nabla{\times} \bm{f}-\bm{\Pi}_K^{p-2}\nabla{\times} \bm{f}\|_{\widetilde{K}}\Big)\|\bm{n}_F{\times} \jump{\nabla_h\bm{u}_h \bm{n}_F}\|_{F}.
		\end{split}
		\label{eq:error-cell-f}
	\end{equation}
	Combining \eqref{eq:edge-normal-expression}, \eqref{eq:curlcurl-r} and \eqref{eq:error-cell-f} yields
	\begin{equation}
		\begin{split}
			&\mu (\bm{n}_F{\times} \jump{\nabla_h\bm{u}_h \bm{n}_F}, \bm{r}b_{\widetilde{F}}^2 )_F \\
&\leq C(p) \Big(h_F^{-\frac12}\|\mu\nabla_h\bm{e}\|_{\widetilde{K}}
			+h_F^{3/2}\|\nabla{\times} \bm{f}-\bm{\Pi}_K^{p-2}\nabla{\times} \bm{f}\|_{\widetilde{K}}\Big) \|\bm{n}_F{\times} \jump{\nabla_h\bm{u}_h \bm{n}_F}\|_{F}.
		\end{split}
		\label{eq:edge-normal-bubble}
	\end{equation}
	Using face version of \eqref{eq:inverse-bubble} with $\beta=2$ and $\alpha=0$, we have
	\begin{align*}
		\mu\|\bm{n}_F{\times} \jump{\nabla_h\bm{u}_h \bm{n}_F} \|_{F}^2\leq C(p) \mu (\bm{n}_F{\times} \jump{\nabla_h\bm{u}_h \bm{n}_F}, b_{\widetilde{F}}^2\bm{r})_F,
	\end{align*}
which combined with \eqref{eq:edge-normal-bubble} yields
	\begin{align}
		\mu h_F^{\frac12} \|\bm{n}_F{\times} \jump{\nabla_h\bm{u}_h \bm{n}_F} \|_{F}\leq C\Big(\|\mu\nabla_h\bm{e}\|_{\widetilde{K}}+h_F^{2}\|\nabla{\times} \bm{f}-\bm{\Pi}_K^{p-2}\nabla{\times} \bm{f}\|_{\widetilde{K}}\Big).\label{eq:control-nt}
	\end{align}

   (ii) Now we estimate the second term on the right-hand side of \eqref{eq:normal-tangential-decompose}. To this end, we decompose $\bm{u}_h$ into its normal and tangential components
    \[
\bm{u}_h|_F = (\bm{u}_h{\cdot}\bm{n}_F)\bm{n}_F + \bm{u}_t, \text{where}\;\bm{u}_t = \bm{n}_F{\times} (\bm{u}_h{\times} \bm{n}_F) \qquad \forall F\in \Fint\cap \partial K.
\]
Let  $\partial_{n} = \bm{n}_F{\cdot}\nabla$ denote the normal derivative, and let $\tdiv_\Gamma$ be the surface divergence operator. Since $\nabla{\cdot}\bm{u}_h=0$ pointwise and $\bm{n}_F$ is a constant unit normal, applying the divergence operator to the above decomposition yields the following result.
\begin{equation}
\begin{split}
0=\nabla {\cdot} \bm{u}_h|_F =\partial_n (\bm{u}_h{\cdot}\bm{n}_F)+(\bm{u}_h{\cdot}\bm{n}_F) \nabla{\cdot} \bm{n}_F +\tdiv_\Gamma \bm{u}_t-\bm{u}_t{\cdot} \partial_{n} \bm{n}_F =\partial_{n} (\bm{u}_h{\cdot}\bm{n}_F)+\tdiv_\Gamma \bm{u}_t.
\end{split}
\label{eq:div-free-split}
\end{equation}
Applying the product rule to $\partial_{n} (\bm{u}_h{\cdot} \bm{n}_F)$, we get
\[
\partial_{n} (\bm{u}_h{\cdot} \bm{n})
= (\partial_{n} \bm{u}_h){\cdot} \bm{n} + \bm{u}_h {\cdot} \partial_{n} \bm{n}_F=(\partial_{n} \bm{u}_h){\cdot} \bm{n}_F,
\]
which yields
\begin{align*}
\jump{(\nabla \bm{u}_h\,\bm{n}){\cdot} \bm{n}}_F
= -\tdiv_\Gamma \jump{(\bm{n}{\times} (\bm{u}_h{\times} \bm{n}))}_F .
\end{align*}
Then an application of the inverse equality \eqref{eq:inverse} on each face $F\in \Fint$ gives
\begin{align}
\|\jump{(\nabla \bm{u}_h\,\bm{n}){\cdot} \bm{n}}\|_{F}\leq C(p) h_F^{-1}\|\jump{(\bm{n}{\times} (\bm{u}_h{\times} \bm{n}))} \|_{F} \leq C(p) h_F^{-1}\|\jump{\bm{n}{\times}\bm{u}_h }\|_{F}.\label{eq:control-nn}
\end{align}
Summing over all the faces $F\in \partial K\backslash \partial \Omega$ and combining \eqref{eq:normal-tangential-decompose}, \eqref{eq:control-nt} and \eqref{eq:control-nn} yields the proof of  \eqref{eq:error-normal-flux}.

	(4) Proof of \eqref{eq:error-delta-uh}. Choosing $\bm{w}:=b_{\widetilde{F}}^2 \bm{r}$ as the test function in the error equation \eqref{eq:error-fourth}, where $\bm{r}:=E(\bm{n}{\times} \jump{\Delta_h\bm{u}_h})$ and recall that $b_{\widetilde{F}}$ is the standard face bubble function associated with the vertices of $F$ on each tetrahedron forming the kite $\widetilde{K}$. We deduce that
 	\begin{equation}
		\begin{split}
			(\mu \bm{n}_F{\times} \jump{\Delta_h\bm{u}_h},\bm{w})_F&=	(\mu \nabla_h \bm{e}, \nabla(\nabla{\times} \bm{w}))_{\widetilde{K}}+(\nabla_h{\times} (\bm{f}+\mu \Delta_h\bm{u}_h), \bm{w})_{\widetilde{K}}\\
&+(\mu \jump{\bm{n}{\times} (\nabla_h\bm{u}_h \bm{n})}, \nabla{\times} \bm{w})_F
=:S_1+S_2+S_3.
		\end{split}
		\label{eq:deltauh-face}
	\end{equation}
Proceeding in a fashion analogous to \eqref{eq:curlcurl-r}, we apply the Cauchy-Schwarz inequality to obtain
	\begin{align*}
		S_1&\leq  \|\mu \nabla_h \bm{e}\|_{\widetilde{K}} \|\nabla(\nabla{\times} \bm{w})\|_{\widetilde{K}}\leq C(p)  h_F^{-3/2}\|\mu \nabla_h\bm{e}\|_{\widetilde{K}}\|\bm{n}_F{\times} \jump{\Delta_h\bm{u}_h}\|_{F}.
	\end{align*}
	The properties of bubble function give $\|\bm{w}\|_{\widetilde{K}}\leq C \|\bm{r}\|_{\widetilde{K}}$, this combined with  \eqref{eq:residual-curlf} and \eqref{eq:extension-stability} leads to
	\begin{align*}
		S_2&\leq \|\nabla_h{\times} (\bm{f}+\mu \Delta_h\bm{u}_h)\|_{\widetilde{K}} \|\bm{w}\|_{\widetilde{K}}\leq C \|\nabla_h{\times} (\bm{f}+\mu \Delta_h\bm{u}_h)\|_{\widetilde{K}} \|\bm{r}\|_{\widetilde{K}}\\
		&\leq C(p) \Big(\frac{1}{h_F^{3/2}} \mu \|\nabla_h\bm{e}\|_{\widetilde{K}}+h_F^{ \frac12}\|\nabla{\times} \bm{f}-\bm{\Pi}_K^{p-2}\nabla{\times} \bm{f}\|_{\widetilde{K}}\Big)\|\bm{n}_F{\times} \jump{\Delta_h\bm{u}_h}\|_{F}.
	\end{align*}
	Now we estimate $S_3$.
	The chain rule gives
	\begin{align*}
		\nabla{\times} \bm{w}
		=\nabla {\times} (b_{\widetilde{F}}^2\bm{r})
&=b_{\widetilde{F}}^2 (\nabla {\times} \bm{r})
		+\nabla (b_{\widetilde{F}}^2) {\times} \bm{r}=b_{\widetilde{F}}
		\Big(
		b_{\widetilde{F}}(\nabla {\times} \bm{r})
		+2(\nabla b_{\widetilde{F}}){\times} \bm{r}
		\Big).
	\end{align*}
	Then, we can rewrite $S_3$ as
	\begin{align*}
		S_3&=(\mu \jump{\bm{n}{\times} (\nabla_h\bm{u}_h \bm{n})}, \nabla{\times} \bm{w})_F \\
&=(\mu \jump{\bm{n}{\times} (\nabla_h\bm{u}_h \bm{n})}b_{\widetilde{F}}, b_{\widetilde{F}}(\nabla {\times} \bm{r}))_F+(\mu \jump{\bm{n}{\times} (\nabla_h\bm{u}_h \bm{n})}b_{\widetilde{F}}, 2(\nabla b_{\widetilde{F}}){\times} \bm{r})_F=:S_{31}+S_{32}.
	\end{align*}
	The Cauchy-Schwarz inequality gives
	\begin{align*}
		S_{31}\leq \mu\|b_{\widetilde{F}}^{3/2}\jump{\bm{n}{\times} (\nabla_h\bm{u}_h \bm{n})}\|_{F} \|b_{\widetilde{F}}^{1/2}(\nabla {\times} \bm{r})\|_{F}.
	\end{align*}
	The trace inequality \eqref{eq: discrete trace}, inverse inequality \eqref{eq:inverse}, the stability of extension operator \eqref{eq:extension-stability} and the shape-regularity of the mesh  yield
	\begin{align*}
		\|b_{\widetilde{F}}^{1/2}\nabla {\times} \bm{r}\|_{F}\leq  \frac{C(p)}{h_K^{1/2}} \|\nabla{\times} \bm{r}\|_{\widetilde{K}}\leq C \frac{C(p)}{h_K^{3/2}} \|\bm{r}\|_{\widetilde{K}}\leq  \frac{C(p)}{h_F} \|\bm{n}_F{\times} \jump{\Delta_h \bm{u}_h}\|_{F}.
	\end{align*}
	This combined with \eqref{eq:control-nt} gives
	\begin{align*}
		S_{31}\leq C(p)h_F^{-3/2} \Big(\|\mu\nabla_h\bm{e}\|_{\widetilde{K}}+h_K\|\nabla{\times} \bm{f}-\bm{\Pi}_K^{p-2}\nabla{\times} \bm{f}\|_{\widetilde{K}}\Big) \| \bm{n}_F{\times} \jump{\Delta_h\bm{u}_h}\|_{F}.
	\end{align*}
	The Cauchy-Schwarz inequality, $ b_{\widetilde{F}}$, \eqref{eq: discrete trace} and \eqref{eq:control-nt} yield
	\begin{align*}
		S_{32}&\leq 2\|\mu \jump{\bm{n}{\times} (\nabla_h\bm{u}_h \bm{n})} \|(\nabla b_{\widetilde{F}}){\times} \bm{r}\|_{F}\\
        &\leq C(p)h_K^{-2} \Big(\|\mu\nabla_h\bm{e}\|_{\widetilde{K}}  + h_K^2\|\nabla{\times} \bm{f}-\bm{\Pi}_K^{p-2}\nabla{\times} \bm{f}\|_{\widetilde{K}} \Big)\|\bm{r}\|_{\widetilde{K}}\\
		&\leq C(p)h_K^{-3/2} \Big(\|\mu\nabla_h\bm{e}\|_{\widetilde{K}}  +h_K^2\|\nabla{\times} \bm{f}-\bm{\Pi}_K^{p-2}\nabla{\times} \bm{f}\|_{\widetilde{K}}\Big)\| \bm{n}_F{\times} \jump{\Delta_h\bm{u}_h}\|_{F}.
	\end{align*}
	Combining the above bounds with \eqref{eq:deltauh-face}, we obtain
	\begin{align}
		\mu	\| \bm{n}_F{\times} \jump{\Delta_h \bm{u}_h}b_{\widetilde{F}}\|_{F}\leq \frac{C(p)}{h_K^\frac32}\Big( \|\mu \nabla_h\bm{e}\|_{\widetilde{K}}+h_K^2\|\nabla{\times} \bm{f}-\bm{\Pi}_K^{p-2}\nabla{\times} \bm{f}\|_{\widetilde{K}}  \Big) {\times} \| \bm{n}_F\times\jump{\Delta_h\bm{u}_h}\|_{F}.\label{eq:deltauh-bubble}
	\end{align}
Invoking	\eqref{eq:inverse-bubble} with $\beta=2$ and $\alpha=0$, we infer that
	\begin{align*}
		\| \bm{n}_F{\times} \jump{\Delta_h \bm{u}_h}\|_{F}^2\leq  C	\| \bm{n}_F{\times} \jump{\Delta_h \bm{u}_h}b_{\widetilde{F}}\|_{F}^2.
	\end{align*}
	This and \eqref{eq:deltauh-bubble} give \eqref{eq:error-delta-uh}.

(5) Proof of \eqref{eq:error-tangential}.  An application of the inverse inequality \eqref{eq:inverse} on each face
$F\in\Fall$, combined with the fact that
$\bm{u}_h \in \bm{H}(\tdiv;\Omega)$, yields the desired result:
\begin{equation}
\left( \frac{h_F}{p}\right)
\|\bm{n}{\times} \jump{\nabla_h\bm{u}_h}\|_{F}^2
\leq
C\left( \frac{p^3}{h_F}\right)
\|\jump{\bm{u}_h}\|_{F}^2
=
C(p)\left( \frac{p^2}{h_F}\right)
\|\bm{n}_F{\times}\jump{\bm{u}_h}\|_{F}^2.
\end{equation}
Finally, the proof is completed by combining~\eqref{eq:error-cell-residual}--\eqref{eq:error-tangential}.
\end{proof}

\bibliographystyle{siam}
\bibliography{references}

@article {CostabelMcIntosh2010,
    AUTHOR = {Costabel, Martin and McIntosh, Alan},
     TITLE = {On {B}ogovski\u i\ and regularized {P}oincar\'e{} integral
              operators for de {R}ham complexes on {L}ipschitz domains},
   JOURNAL = {Math. Z.},
  FJOURNAL = {Mathematische Zeitschrift},
    VOLUME = {265},
      YEAR = {2010},
    NUMBER = {2},
     PAGES = {297--320},
      ISSN = {0025-5874,1432-1823},
   MRCLASS = {58J10 (35B65 35S05 47G30)},
  MRNUMBER = {2609313},
MRREVIEWER = {Horst\ Heck},
       DOI = {10.1007/s00209-009-0517-8},
       URL = {https://doi.org/10.1007/s00209-009-0517-8},
}

@article {LedererSchoberl18,
    AUTHOR = {Lederer, Philip L. and Sch\"oberl, Joachim},
     TITLE = {Polynomial robust stability analysis for {$H({\rm
              div})$}-conforming finite elements for the {S}tokes equations},
   JOURNAL = {IMA J. Numer. Anal.},
  FJOURNAL = {IMA Journal of Numerical Analysis},
    VOLUME = {38},
      YEAR = {2018},
    NUMBER = {4},
     PAGES = {1832--1860},
      ISSN = {0272-4979,1464-3642},
   MRCLASS = {65N30 (65N12 76D07 76M10 76S05)},
  MRNUMBER = {3867384},
MRREVIEWER = {Hans-Peter\ Helfrich},
       DOI = {10.1093/imanum/drx051},
       URL = {https://doi.org/10.1093/imanum/drx051},
}

@unpublished{chaumontfrelet:hal-05204325,
  TITLE = {{Computable Poincar{\'e}--Friedrichs constants for the $L^p$ de Rham complex over convex domains and domains with shellable triangulations}},
  AUTHOR = {Chaumont-Frelet, T. and Licht, M. W. and Vohral{\'i}k, M.},
  NOTE = {Preprint, \texttt{https://inria.hal.science/hal-05204325}},
  YEAR = {2025},
  HAL_ID = {hal-05204325},
  HAL_VERSION = {v1},
}

@article {LeePre:79,
    AUTHOR = {Lee, D. T. and Preparata, F. P.},
     TITLE = {An optimal algorithm for finding the kernel of a polygon},
   JOURNAL = {J. Assoc. Comput. Mach.},
  FJOURNAL = {Journal of the Association for Computing Machinery},
    VOLUME = {26},
      YEAR = {1979},
    NUMBER = {3},
     PAGES = {415--421},
       optDOI = {10.1145/322139.322142},
       optURL = {https://doi-org.extranet.enpc.fr/10.1145/322139.322142},
}

@unpublished{dong:hal-05498158,
  TITLE = {{$hp$-a posteriori error estimates for hybrid high-order methods applied to biharmonic problems}},
  AUTHOR = {Dong, Zhaonan and Ern, Alexandre and Wadhawan, Tanvi},
  URL = {https://hal.science/hal-05498158},
  NOTE = {working paper or preprint},
  YEAR = {2026},
  MONTH = Feb,
  HAL_ID = {hal-05498158},
  HAL_VERSION = {v1},
}

@book {Ern_Guermond_FEs_I_2021,
	AUTHOR = {Ern, A. and Guermond, J.-L.},
	TITLE = {Finite Elements {I}: {A}pproximation and Interpolation},
	VOLUME = {72},
	SERIES = {Texts in Applied Mathematics},
	PUBLISHER = {Springer Nature},
	ADDRESS = {Cham, Switzerland},
	YEAR = {2021},
	optPAGES = {},
}

@article {CarGraTra24,
    AUTHOR = {Carstensen, Carsten and Gr\"a{\ss}le, Benedikt and Tran, Ngoc
              Tien},
     TITLE = {Adaptive hybrid high-order method for guaranteed lower
              eigenvalue bounds},
   JOURNAL = {Numer. Math.},
  FJOURNAL = {Numerische Mathematik},
    VOLUME = {156},
      YEAR = {2024},
    NUMBER = {3},
     PAGES = {813--851},
      ISSN = {0029-599X,0945-3245},
   MRCLASS = {65N12 (65N30 65Y20)},
  MRNUMBER = {4755183},
MRREVIEWER = {Martin\ Vohral\'ik},
       DOI = {10.1007/s00211-024-01407-w},
       URL = {https://doi.org/10.1007/s00211-024-01407-w},
}

@ARTICLE
{babuskasurihpversionFEMwithquasiuniformmesh,
  title={The $hp$ version of the Finite Element Method with Quasiuniform Meshes},
  author={Babu\v{s}ka, I. and Suri, M.},
  journal={ESAIM Math. Model. Numer. Anal.},
  volume={21},
  number={2},
  pages={199--238},
  year={1987},
  publisher={EDP Sciences}
}

@ARTICLE
{BabuSurioptimalconvergenceestimatepmethods,
  title={The optimal convergence rate of the $p$-version of the Finite Element Method},
  author={Babu\v{s}ka, I. and Suri, M.},
  journal={SIAM J. Numer. Anal.},
  volume={24},
  number={4},
  pages={750--776},
  year={1987},
  publisher={SIAM}
}

@book{Girault86,
	author = {V. Girault and P.-A. Raviart},
	title = {Finite Element Methods for {N}avier-{S}tokes Equations},
	year = {1986},
	publisher ={Springer Berlin, Heidelberg}
}

@article {Warburton03,
    AUTHOR = {Warburton, T. and Hesthaven, J. S.},
     TITLE = {On the constants in {$hp$}-finite element trace inverse
              inequalities},
   JOURNAL = {Comput. Methods Appl. Mech. Engrg.},
  FJOURNAL = {Computer Methods in Applied Mechanics and Engineering},
    VOLUME = {192},
      YEAR = {2003},
    NUMBER = {25},
     PAGES = {2765--2773},
      ISSN = {0045-7825,1879-2138},
   MRCLASS = {65N30},
  MRNUMBER = {1986022},
       DOI = {10.1016/S0045-7825(03)00294-9},
       URL = {https://doi.org/10.1016/S0045-7825(03)00294-9},
}

@article {Hannukainen12,
    AUTHOR = {Hannukainen, Antti and Stenberg, Rolf and Vohral\'ik, Martin},
     TITLE = {A unified framework for a posteriori error estimation for the
              {S}tokes problem},
   JOURNAL = {Numer. Math.},
  FJOURNAL = {Numerische Mathematik},
    VOLUME = {122},
      YEAR = {2012},
    NUMBER = {4},
     PAGES = {725--769},
      ISSN = {0029-599X,0945-3245},
   MRCLASS = {65N15 (65N08 65N30 76D07 76M10 76M12)},
  MRNUMBER = {2995179},
MRREVIEWER = {Muthusamy\ Vanninathan},
       DOI = {10.1007/s00211-012-0472-x},
       URL = {https://doi.org/10.1007/s00211-012-0472-x},
}

@article {Bacuta16,
    AUTHOR = {Bacuta, Constantin},
     TITLE = {Sharp stability and approximation estimates for symmetric
              saddle point systems},
   JOURNAL = {Appl. Anal.},
  FJOURNAL = {Applicable Analysis. An International Journal},
    VOLUME = {95},
      YEAR = {2016},
    NUMBER = {1},
     PAGES = {226--237},
      ISSN = {0003-6811,1563-504X},
   MRCLASS = {65J10 (65N30 74S05)},
  MRNUMBER = {3426198},
MRREVIEWER = {Marius\ Ghergu},
       DOI = {10.1080/00036811.2014.1002483},
       URL = {https://doi.org/10.1080/00036811.2014.1002483},
}

@article{Melenk05,
	author = {Melenk, J. M.},
	title = {hp-Interpolation of Nonsmooth Functions and an Application to hp-A posteriori Error Estimation},
	journal = {SIAM J. Numer. Anal.},
	volume = {43},
	number = {1},
	pages = {127--155},
	year = {2005}
}

@article {KarMel2015,
    AUTHOR = {Karkulik, M. and Melenk, J. M.},
     TITLE = {Local high-order regularization and applications to
              {$hp$}-methods},
   JOURNAL = {Comput. Math. Appl.},
  FJOURNAL = {Computers \& Mathematics with Applications. An International
              Journal},
    VOLUME = {70},
      YEAR = {2015},
    NUMBER = {7},
     PAGES = {1606--1639},
      ISSN = {0898-1221,1873-7668},
   MRCLASS = {65N30 (65N15 65N38)},
  MRNUMBER = {3396963},
       DOI = {10.1016/j.camwa.2015.06.026},
       URL = {https://doi.org/10.1016/j.camwa.2015.06.026},
}

@article{dong:hal-04720237,
  TITLE = {{$hp$-error analysis of mixed-order hybrid high-order methods for elliptic problems on simplicial meshes}},
  AUTHOR = {Dong, Zhaonan and Ern, Alexandre},
  URL = {https://inria.hal.science/hal-04720237},
  VOLUME = {to appear},
  JOURNAL = {{Numerische Mathematik}},
  PUBLISHER = {{Springer Verlag}},
  YEAR = {2026},
  HAL_ID = {hal-04720237},
  HAL_VERSION = {v3},
}

@article{dong:hal-05213366,
  TITLE = {{$\boldsymbol {H}  (\textbf{curl})$-reconstruction of piecewise polynomial fields with application to $hp$-a posteriori nonconforming error analysis for Maxwell's equations}},
  AUTHOR = {Dong, Z. and Ern, A.},
  URL = {https://inria.hal.science/hal-05213366},
	VOLUME = {to appear},
  JOURNAL = {{SIAM J. Numer. Anal.}},
  PUBLISHER = {{Society for Industrial and Applied Mathematics}},
  YEAR = {2026},
  HAL_ID = {hal-05213366},
  HAL_VERSION = {v3},
}

@article{Melenk01,
	title={On residual-based a posteriori error estimation in hp-{FEM}},
	author={Melenk, J. M. and Wohlmuth, B. I.},
	journal={Adv. Comput. Math.},
	volume={15},
	number={1},
	pages={311--331},
	year={2001},
	publisher={Springer}
}

@book{Schwab98,
	title={p-and hp-finite element methods. Theory and applications in solid and fluid mechanics},
	author={C. Schwab},
	year={1998},
	publisher={Clarendon Press, Oxford}
}

@article {AinsworkthParker1,
    AUTHOR = {Ainsworth, Mark and Parker, Charles},
     TITLE = {Mass conserving mixed {$hp$}-{FEM} approximations to {S}tokes
              flow. {P}art {I}: {U}niform stability},
   JOURNAL = {SIAM J. Numer. Anal.},
  FJOURNAL = {SIAM Journal on Numerical Analysis},
    VOLUME = {59},
      YEAR = {2021},
    NUMBER = {3},
     PAGES = {1218--1244},
      ISSN = {0036-1429,1095-7170},
   MRCLASS = {65N30 (35Q30 65N12 76D07 76M10)},
  MRNUMBER = {4256089},
       DOI = {10.1137/20M1359109},
       URL = {https://doi.org/10.1137/20M1359109},
}

@article{Chernov12,
	author = {A. Chernov},
	journal = {Math. Comp.},
	number = {278},
	pages = {765--787},
	publisher = {American Mathematical Society},
	title = {OPTIMAL CONVERGENCE ESTIMATES FOR THE TRACE OF THE POLYNOMIAL ${L}^2$-PROJECTION OPERATOR ON A SIMPLEX},
	urldate = {2025-10-09},
	volume = {81},
	year = {2012}
}

@article{Melenkhp,
	title={On commuting $p-$version projection-based interpolation on tetrahedra},
	author={Melenk, J. M. and Rojik, C.},
	year={2020},
	pages = {45--87},
	Journal = {Math. Comp.},
	publisher={Compet}
}

@article{GuzmanAbner21,
	title = {Estimation of the continuity constants for Bogovski\u{i} and regularized Poincar\'{e} integral operators},
	journal = {J. Math. Anal. Appl.},
	volume = {502},
	number = {1},
	pages = {125246},
	year = {2021},
	author = {J. Guzm\'{a}n and A. J. Salgado},
}

@article{Dong21,
	author = {Dong, Z. and Mascotto, L. and Sutton, O. J.},
	title = {Residual-Based A Posteriori Error Estimates for \$hp\$-Discontinuous {G}alerkin Discretizations of the Biharmonic Problem},
	journal = {SIAM J. Numer. Anal.},
	volume = {59},
	number = {3},
	pages = {1273--1298},
	year = {2021}
}

@article{Ahmed21,
	author = {Ahmed, N. and Barrenechea, G. R. and Burman, E. and Guzm\'{a}n, J. and Linke, A. and Merdon, C.},
	title = {A Pressure-Robust Discretization of {O}seen's Equation Using Stabilization in the Vorticity Equation},
	journal = {SIAM J. Numer. Anal.},
	volume = {59},
	number = {5},
	pages = {2746-2774},
	year = {2021}
}

@book {Boffibook,
    AUTHOR = {Boffi, Daniele and Brezzi, Franco and Fortin, Michel},
     TITLE = {Mixed finite element methods and applications},
    SERIES = {Springer Series in Computational Mathematics},
    VOLUME = {44},
 PUBLISHER = {Springer, Heidelberg},
      YEAR = {2013},
     PAGES = {xiv+685},
      ISBN = {978-3-642-36518-8; 978-3-642-36519-5},
   MRCLASS = {65-02 (65M60 65N30)},
  MRNUMBER = {3097958},
MRREVIEWER = {Beny\ Neta},
       DOI = {10.1007/978-3-642-36519-5},
       URL = {https://doi.org/10.1007/978-3-642-36519-5},
}

@book{Verfurth2013posteriori,
	title={A posteriori error estimation techniques for finite element methods},
	author={Verf{\"u}rth, R.},
	year={2013},
	publisher={Oxford University Press},
	address={Oxford}
}

@article{VACMG17,
	author = {John, Volker and Linke, Alexander and Merdon, Christian and Neilan, Michael and Rebholz, Leo G.},
	title = {On the Divergence Constraint in Mixed Finite Element Methods for Incompressible Flows},
	journal = {SIAM Review},
	volume = {59},
	number = {3},
	pages = {492--544},
	year = {2017}
}

@article {LCJ19,
	AUTHOR = {Lederer, Philip Lukas and Merdon, Christian and Sch\"oberl,
	Joachim},
	TITLE = {Refined a posteriori error estimation for classical and
	pressure-robust {S}tokes finite element methods},
	JOURNAL = {Numer. Math.},
	FJOURNAL = {Numerische Mathematik},
	VOLUME = {142},
	YEAR = {2019},
	NUMBER = {3},
	PAGES = {713--748}
}

@article {LC22,
	AUTHOR = {Lederer, Philip L. and Merdon, Christian},
	TITLE = {Guaranteed upper bounds for the velocity error of
	pressure-robust {S}tokes discretisations},
	JOURNAL = {J. Numer. Math.},
	FJOURNAL = {Journal of Numerical Mathematics},
	VOLUME = {30},
	YEAR = {2022},
	NUMBER = {4},
	PAGES = {267--294},
	ISSN = {1570-2820,1569-3953}
}

@article {JLJ20,
	AUTHOR = {Gopalakrishnan, Jay and Lederer, Philip L. and Sch\"oberl,
	Joachim},
	TITLE = {A mass conserving mixed stress formulation for the {S}tokes
	equations},
	JOURNAL = {IMA J. Numer. Anal.},
	FJOURNAL = {IMA Journal of Numerical Analysis},
	VOLUME = {40},
	YEAR = {2020},
	NUMBER = {3},
	PAGES = {1838--1874}
}

@article{Schoberl14,
  title={C++ 11 implementation of finite elements in NGSolve},
  author={Sch{\"o}berl, Joachim},
  journal={Institute for analysis and scientific computing, Vienna University of Technology},
  volume={30},
  year={2014}
}

@article {Zhang05,
	AUTHOR = {Zhang, Shangyou},
	TITLE = {A new family of stable mixed finite elements for the 3{D}
	{S}tokes equations},
	JOURNAL = {Math. Comp.},
	FJOURNAL = {Mathematics of Computation},
	VOLUME = {74},
	YEAR = {2005},
	NUMBER = {250},
	PAGES = {543--554}
}

@article {LLMJ17,
	AUTHOR = {Lederer, Philip L. and Linke, Alexander and Merdon, Christian
	and Sch\"oberl, Joachim},
	TITLE = {Divergence-free reconstruction operators for pressure-robust
	{S}tokes discretizations with continuous pressure finite
	elements},
	JOURNAL = {SIAM J. Numer. Anal.},
	FJOURNAL = {SIAM Journal on Numerical Analysis},
	VOLUME = {55},
	YEAR = {2017},
	NUMBER = {3},
	PAGES = {1291--1314}
}

@article {ZCL20,
	AUTHOR = {Zhao, Lina and Chung, Eric and Lam, Ming Fai},
	TITLE = {A new staggered {DG} method for the {B}rinkman problem robust
	in the {D}arcy and {S}tokes limits},
	JOURNAL = {Comput. Methods Appl. Mech. Engrg.},
	FJOURNAL = {Computer Methods in Applied Mechanics and Engineering},
	VOLUME = {364},
	YEAR = {2020},
	PAGES = {112986, 18}
}

@article{KS14,
	author = {Kanschat, Guido and Sharma, Natasha},
	title = {Divergence-Conforming Discontinuous {G}alerkin Methods and ${C}^0$ Interior Penalty Methods},
	journal = {SIAM J. Numer. Anal.},
	volume = {52},
	number = {4},
	pages = {1822--1842},
	year = {2014}
}

@article{LinkeMerdon16,
	author  = {Alexander Linke and Christian Merdon},
	title   = {Pressure-robustness and discrete {H}elmholtz projectors in mixed finite element methods for the incompressible {N}avier--{S}tokes equations},
	JOURNAL = {Comput. Methods Appl. Mech. Engrg.},
FJOURNAL = {Computer Methods in Applied Mechanics and Engineering},
	volume  = {311},
	pages   = {304--326},
	year    = {2016}
}

@article{Verfurth89,
	author  = {Verf{\"u}rth, R{\"u}diger},
	title   = {A Posteriori Error Estimators for the {S}tokes Equations},
	journal = {Numer. Math.},
	volume  = {55},
	pages   = {309--325},
	year    = {1989}
}

@article{DorflerAinsworth2005,
	author  = {Willy D{\"o}rfler and Mark Ainsworth},
	title   = {Reliable a Posteriori Error Control for Non-Conforming Finite Element Approximation of {S}tokes Flow},
	journal = {Math. Comp.},
	volume   = {74},
	number   = {252},
	pages    = {1599--1619},
	year     = {2005}
}

@article{CarstensenFunken2001,
	author       = {Carstensen, Carsten and Funken, Stefan A.},
	title        = {A Posteriori Error Control in Low-Order Finite Element
	Discretisations of Incompressible Stationary Flow Problems},
	journal= {Math. Comp.},
	year         = {2001},
	volume       = {70},
	number       = {236},
	pages        = {1353--1381}
}

@article{CarstensenKimPark2011,
	author       = {Carstensen, Carsten and Kim, Dongho and Park, Eun-Jae},
	title        = {A Priori and A Posteriori Pseudostress-Velocity Mixed
	Finite Element Error Analysis for the {S}tokes Problem},
	journal = {SIAM J. Numer. Anal.},
	year         = {2011},
	volume       = {49},
	number       = {6},
	pages        = {2501--2523}
}

@article{DariDuranPadra1995,
	author       = {Dari, Enzo and Dur{\'a}n, Ricardo G. and Padra, Claudio},
	title        = {Error Estimators for Nonconforming Finite Element Approximations of the {S}tokes Problem},
	journal = {Math. Comp.},
	year         = {1995},
	volume       = {64},
	number       = {211},
	pages        = {1017--1033}
}

@article{AinsworthOden1997,
	author       = {Ainsworth, Mark and Oden, J. Tinsley},
	title        = {A Posteriori Error Estimators for the {S}tokes and {O}seen Equations},
	journal = {SIAM J. Numer. Anal.},
	year         = {1997},
	volume       = {34},
	number       = {1},
	pages        = {228--245}
}

@article {CG26,
	AUTHOR = {Carstensen, Carsten and Gr\"assle, Benedikt},
	TITLE = {Adaptive {M}orley {FEM} for 2{D} stationary {N}avier-{S}tokes},
	JOURNAL = {Math. Comp.},
	FJOURNAL = {Mathematics of Computation},
	VOLUME = {95},
	YEAR = {2026},
	NUMBER = {358},
	PAGES = {613--645},
	ISSN = {0025-5718,1088-6842},
	MRCLASS = {65N30 (65N12 65N50)},
	MRNUMBER = {4999538},
	DOI = {10.1090/mcom/4069},
	URL = {https://doi.org/10.1090/mcom/4069}
}

@article {GuzmanNeilan2014,
	AUTHOR = {Guzm\'an, Johnny and Neilan, Michael},
	TITLE = {Conforming and divergence-free {S}tokes elements on general triangular meshes},
	JOURNAL = {Math. Comp.},
	FJOURNAL = {Mathematics of Computation},
	VOLUME = {83},
	YEAR = {2014},
	NUMBER = {285},
	PAGES = {15--36},
	ISSN = {0025-5718,1088-6842},
	MRCLASS = {65N30 (65N12 76D07 76M10)},
	MRNUMBER = {3120580},
	MRREVIEWER = {Marius\ Ghergu},
	DOI = {10.1090/S0025-5718-2013-02753-6},
	URL = {https://doi.org/10.1090/S0025-5718-2013-02753-6}
}

@article {GuzmanNeilan2014b,
	AUTHOR = {Guzm\'an, Johnny and Neilan, Michael},
	TITLE = {Conforming and divergence-free {S}tokes elements in three
	dimensions},
	JOURNAL = {IMA J. Numer. Anal.},
	FJOURNAL = {IMA Journal of Numerical Analysis},
	VOLUME = {34},
	YEAR = {2014},
	NUMBER = {4},
	PAGES = {1489--1508},
	ISSN = {0272-4979,1464-3642},
	MRCLASS = {65N30 (65N12 76D07)},
	MRNUMBER = {3269433},
	MRREVIEWER = {David\ Maltese},
	DOI = {10.1093/imanum/drt053},
	URL = {https://doi.org/10.1093/imanum/drt053},
}

@article {EggerWaluga2013,
	AUTHOR = {Egger, Herbert and Waluga, Christian},
	TITLE = {{$hp$} analysis of a hybrid {DG} method for {S}tokes flow},
	JOURNAL = {IMA J. Numer. Anal.},
	FJOURNAL = {IMA Journal of Numerical Analysis},
	VOLUME = {33},
	YEAR = {2013},
	NUMBER = {2},
	PAGES = {687--721},
	ISSN = {0272-4979,1464-3642},
	MRCLASS = {65N30 (76D07)},
	MRNUMBER = {3047948},
	DOI = {10.1093/imanum/drs018},
	URL = {https://doi.org/10.1093/imanum/drs018}
}

@article {LiuLiuPego2007,
    AUTHOR = {Liu, Jian-Guo and Liu, Jie and Pego, Robert L.},
     TITLE = {Stability and convergence of efficient {N}avier-{S}tokes
              solvers via a commutator estimate},
   JOURNAL = {Comm. Pure Appl. Math.},
  FJOURNAL = {Communications on Pure and Applied Mathematics},
    VOLUME = {60},
      YEAR = {2007},
    NUMBER = {10},
     PAGES = {1443--1487},
      ISSN = {0010-3640,1097-0312},
   MRCLASS = {76D05 (35Q30 65M06 76M25)},
  MRNUMBER = {2342954},
MRREVIEWER = {Jean-Luc\ Guermond},
       DOI = {10.1002/cpa.20178},
       URL = {https://doi.org/10.1002/cpa.20178},
}

@article {ZPC22,
    AUTHOR = {Zhao, Lina and Park, Eun-Jae and Chung, Eric},
     TITLE = {A pressure robust staggered discontinuous {G}alerkin method
              for the {S}tokes equations},
   JOURNAL = {Comput. Math. Appl.},
  FJOURNAL = {Computers \& Mathematics with Applications. An International
              Journal},
    VOLUME = {128},
      YEAR = {2022},
     PAGES = {163--179},
      ISSN = {0898-1221,1873-7668},
   MRCLASS = {65N30 (35Q30 76M30)},
  MRNUMBER = {4504492},
MRREVIEWER = {Sarvesh\ Kumar},
       DOI = {10.1016/j.camwa.2022.10.019},
       URL = {https://doi.org/10.1016/j.camwa.2022.10.019}
}
\end{document}